\documentclass[12pt,reqno]{article}

\usepackage[usenames]{color}

\usepackage[colorlinks=true,
linkcolor=webgreen,
filecolor=webbrown,
citecolor=webgreen,
urlcolor=black]{hyperref}

\definecolor{webgreen}{rgb}{0,.5,0}
\definecolor{webbrown}{rgb}{.6,0,0}

\usepackage{fullpage}
\usepackage{float}

\usepackage{graphics,amsmath,amssymb}
\usepackage{graphicx}
\usepackage{amsthm}
\usepackage{amsfonts}
\usepackage{latexsym}
\usepackage{epsf}
\usepackage[dvipsnames]{xcolor}

\DeclareMathOperator{\re}{Re}
\DeclareMathOperator{\im}{Im}

\begin{document}

\theoremstyle{plain}
\newtheorem{theorem}{Theorem}
\newtheorem{corollary}[theorem]{Corollary}
\newtheorem{lemma}[theorem]{Lemma}
\newtheorem{proposition}[theorem]{Proposition}

\theoremstyle{definition}
\newtheorem{definition}[theorem]{Definition}
\newtheorem{example}[theorem]{Example}
\newtheorem{conjecture}[theorem]{Conjecture}

\theoremstyle{remark}
\newtheorem{remark}[theorem]{Remark}

\begin{center}
\vskip 1cm{\LARGE\bf The Generalization of Faulhaber's Formula to
\vskip0.1 in Sums of Arbitrary Complex Powers}
\vskip 1cm
\large
Raphael Schumacher\\
Department of Mathematics\\
ETH Zurich\\
R\"amistrasse 101\\
8092 Zurich\\
Switzerland\\
\href{raphael.schumacher@math.ethz.ch}{\tt raphael.schumacher@math.ethz.ch}\\
\end{center}

\vskip .2 in

\begin{abstract}
In this paper we present a generalization of Faulhaber's formula to sums of arbitrary complex powers $m\in\mathbb{C}$. These summation formulas for sums of the form $\sum_{k=1}^{\lfloor x\rfloor}k^{m}$ and $\sum_{k=1}^{n}k^{m}$, where $x\in\mathbb{R}^{+}$ and $n\in\mathbb{N}$, are based on a series acceleration involving Stirling numbers of the first kind. While it is well-known that the corresponding expressions obtained from the Euler-Maclaurin summation formula diverge, our summation formulas are all very rapidly convergent.
\end{abstract}

\section{Introduction}
\label{sec:Introduction}

For two natural numbers $m,n\in\mathbb{N}_{0}$, the Faulhaber formula \cite{1}, given by
\begin{equation}
\sum_{k=0}^{n}k^{m}=\frac{1}{m+1}\sum_{k=0}^{m}(-1)^{k}{m+1\choose k}B_{k}n^{m-k+1},
\end{equation}
where the $B_{k}$'s are the Bernoulli numbers, provides a very efficient way to compute the sum of the $m$-th powers of the first $n$ natural numbers.
This formula was found by Jacob Bernoulli around 1700 and was first proved by Carl Gustav Jacobi in 1834.\\

We will prove a rapidly convergent exact generalization of Faulhaber's formula to finite sums of the form $\sum_{k=1}^{\lfloor x\rfloor}k^{m}$ and $\sum_{k=1}^{n}k^{m}$ for all exponents $m\in\mathbb{C}$. Our key tool will be the so called Weniger transformation \cite[\text{(4.1)}]{2} found by J. Weniger, transforming an inverse power series into an inverse factorial series \cite[\text{(1.1)}]{2}. This transformation of inverse power series was first found by Oskar Schl\"omilch around 1850 \cite{3,4,5} based on earlier works of James Stirling in 1730 \cite{6}.\\
In an expanded form, one of our summation formulas for the sum $\sum_{k=1}^{n}\sqrt{k}$, where $n\in\mathbb{N}$, looks like
\begin{equation}\label{square roots formula}
\begin{split}
\sum_{k=1}^{n}\sqrt{k}&=\frac{2}{3}n^{\frac{3}{2}}+\frac{1}{2}\sqrt{n}-\frac{1}{4\pi}\zeta\left(\frac{3}{2}\right)+\sqrt{n}\sum_{k=1}^{\infty}(-1)^{k+1}\frac{\sum_{l=1}^{k}\frac{(2l-3)!!}{2^{l}(l+1)!}B_{l+1}S_{k}^{(1)}(l)}{(n+1)(n+2)\cdots(n+k)}\\
&=\frac{2}{3}n^{\frac{3}{2}}+\frac{1}{2}\sqrt{n}-\frac{1}{4\pi}\zeta\left(\frac{3}{2}\right)+\frac{\sqrt{n}}{24(n+1)}+\frac{\sqrt{n}}{24(n+1)(n+2)}+\frac{53\sqrt{n}}{640(n+1)(n+2)(n+3)}\\
&\quad+\frac{79\sqrt{n}}{320(n+1)(n+2)(n+3)(n+4)}+\ldots,
\end{split}
\end{equation}
where the $B_{l}$'s are the Bernoulli numbers and $S_{k}^{(1)}(l)$ denotes the Stirling numbers of the first kind.\\
The above identity \eqref{square roots formula} is deduced by setting the variable $x:=n\in\mathbb{N}$ into the more general formula
\begin{equation}
\begin{split}
\sum_{k=1}^{\lfloor x\rfloor}\sqrt{k}&=\frac{2}{3}x^{\frac{3}{2}}-\frac{1}{4\pi}\zeta\left(\frac{3}{2}\right)-\sqrt{x}B_{1}(\{x\})+\sqrt{x}\sum_{k=1}^{\infty}(-1)^{k}\frac{\sum_{l=1}^{k}(-1)^{l}\frac{(2l-3)!!}{2^{l}(l+1)!}S_{k}^{(1)}(l)B_{l+1}(\{x\})}{(x+1)(x+2)\cdots(x+k)}\\
&=\frac{2}{3}x^{\frac{3}{2}}-\frac{1}{4\pi}\zeta\left(\frac{3}{2}\right)+\left(\frac{1}{2}-\{x\}\right)\sqrt{x}+\frac{\left(\frac{1}{4}\{x\}^{2}-\frac{1}{4}\{x\}+\frac{1}{24}\right)\sqrt{x}}{(x+1)}\\
&\quad+\frac{\left(\frac{1}{24}\{x\}^{3}+\frac{3}{16}\{x\}^{2}-\frac{11}{48}\{x\}+\frac{1}{24}\right)\sqrt{x}}{(x+1)(x+2)}+\frac{\left(\frac{1}{64}\{x\}^{4}+\frac{3}{32}\{x\}^{3}+\frac{21}{64}\{x\}^{2}-\frac{7}{16}\{x\}+\frac{53}{640}\right)\sqrt{x}}{(x+1)(x+2)(x+3)}\\
&\quad+\frac{\left(\frac{1}{128}\{x\}^{5}+\frac{19}{256}\{x\}^{4}+\frac{109}{384}\{x\}^{3}+\frac{29}{32}\{x\}^{2}-\frac{977}{768}\{x\}+\frac{79}{320}\right)\sqrt{x}}{(x+1)(x+2)(x+3)(x+4)}+\ldots,
\end{split}
\end{equation}
where this time the $B_{l}(\{x\})$'s are the fractional Bernoulli polynomials and $x\in\mathbb{R}^{+}$ is a positive real number. All other formulas in this article have a similar shape, when we expand them.\\

\vspace{0.2cm}

We have searched our resulting formulas in the literature and on the internet. We could find only two of them, namely equation \eqref{StirlingExpansion} and its analogues for the sums $\sum_{k=1}^{n}\frac{1}{k^{m}}$ with $m\in\mathbb{N}_{\geq2}$, which were already known to Stirling in 1730 \cite{3,7}, and equation \eqref{GregorioFontanaExpansion}, which was obtained by Gregorio Fontana around 1780 \cite{3,8}. Both of these formulas were originally found in another form without the use of Bernoulli numbers.\\
We believe that all other generalized Faulhaber formulas presented in this article are new and that our method to obtain them has not been recognized before.

\section{Definitions and Basic Facts}
\label{sec:Definitions and Basic Facts}

As usual, we denote the floor of $x$ by $\lfloor x\rfloor$ and the fractional part of $x$ by $\{x\}$.\\
The symbol $\mathbb{N}:=\{1,2,3,4,\ldots\}$ denotes the set of natural numbers and $\mathbb{R}^{+}:=\left\{x\in\mathbb{R}:x>0\right\}$ represents the set of positive real numbers. We also set $\mathbb{N}_{0}:=\mathbb{N}\cup\{0\}$ and $\mathbb{R}_{0}^{+}:=\mathbb{R}^{+}\cup\{0\}$. Moreover, we define $\mathbb{H}^{+}:=\{z\in\mathbb{C}:\re(z)>0\}$ and denote by $\zeta(s)$ the Riemann zeta function at the point $s\in\mathbb{C}\setminus\{1\}$. For a complex number $z=re^{i\varphi}\in\mathbb{C}$, we denote by $|z|=r\in\mathbb{R}_{0}^{+}$ its absolute value and by $\varphi=\arg(z)\in(-\pi,\pi]$ its argument or phase. We define for all $\theta\in\mathbb{R}$ the secant function by $\sec(\theta):=\frac{1}{\cos(\theta)}$.\\
The double factorial function for $n\in\mathbb{N}_{0}$ is defined by $n!!:=\prod_{k=0}^{\left\lfloor\frac{n-1}{2}\right\rfloor}(n-2k)$.

\begin{definition}(Pochhammer symbol)\;\cite[\text{p.\;1429}]{2}\\
We define the \emph{Pochhammer symbol} (or rising factorial function) $(z)_{k}$ by 
\begin{equation}
(z)_{k}:=z(z+1)(z+2)(z+3)\cdots(z+k-1)=\frac{\Gamma(z+k)}{\Gamma(z)},
\end{equation}
where $\Gamma(z)$ is the gamma function \cite[\text{(5.2.1), p.\;136}]{9} defined as the meromorphic continuation of the integral
\begin{equation}
\Gamma(z):=\int_{0}^{\infty}e^{-t}t^{z-1}dt\;\;\text{for all $z\in\mathbb{C}$ with $\re(z)>0$}
\end{equation}
to the whole complex plane $\mathbb{C}$.\\
\end{definition}

\begin{definition}(Stirling numbers of the first kind)\;\cite[\text{(A.2), p.\;1437}]{2}, \cite[\text{A008275}]{10}\\
Let $k,l\in\mathbb{N}_{0}$ be two non-negative integers such that $k\geq l\geq0$. We set the \emph{Stirling numbers of the first kind} $S_{k}^{(1)}(l)$ as the connecting coefficients in the identity
\begin{equation}
(z)_{k}=(-1)^{k}\sum_{l=0}^{k}(-1)^{l}S_{k}^{(1)}(l)z^{l},
\end{equation}
where $(z)_{k}$ is the rising factorial function. Furthermore, we set $S_{k}^{(1)}(l)=0$ if $k,l\in\mathbb{N}_{0}$ with $l>k$.
\end{definition}

\begin{definition}(Binomial Coefficients)\;\cite{11}\\
We introduce the \emph{binomial coefficient} ${z\choose s}$ for all $z\in\mathbb{C}$ and $s\in\mathbb{C}$ by \cite[\text{(5) and (11), pp.\;8-9}]{11}
\begin{equation}
\begin{split}
{z\choose s}:&=\frac{\Gamma(z+1)}{\Gamma(s+1)\Gamma(z-s+1)}.
\end{split}
\end{equation}
Moreover, we have for $z\in\mathbb{C}\setminus\{0,-1,-2,-3,\ldots\}$ the following asymptotic expansion as $k\rightarrow\infty$ \cite[\text{(18), p.\;2 and p.\;35}]{11}
\begin{equation}\label{binomial coefficient bound}
{z\choose k}=\frac{(-1)^{k}}{\Gamma(-z)k^{z+1}}+O\left(\frac{1}{k^{z+2}}\right)\;\;\text{for all $z\in\mathbb{C}$ and $k\in\mathbb{N}$}.
\end{equation}
\end{definition}

\begin{definition}(Bernoulli polynomials and Bernoulli numbers)\;\cite{1}, \cite{12}, \cite{13}, \cite{14}\\
We define for $n\in\mathbb{N}_{0}$ the \emph{$n$-th Bernoulli polynomial} $B_{n}(x)$ via the following exponential generating function \cite{1} as
\begin{equation}
\frac{te^{xt}}{e^{t}-1}=\sum_{n=0}^{\infty}\frac{B_{n}(x)}{n!}t^{n}\;\;\forall t\in\mathbb{C}\;\text{with $|t|<2\pi$}.
\end{equation}
We also define the \emph{$n$-th Bernoulli number} $B_{n}$ as the value of the $n$-th Bernoulli polynomial $B_{n}(x)$ at the point $x=0$, that is
\begin{equation}\label{Bernoulli number}
B_{n}:=B_{n}(0).
\end{equation}
It holds for all $n\in\mathbb{N}_{0}$ the explicit formula \cite[\text{Proposition 23.2, p.\;86}]{12}
\begin{equation}\label{Bernoulli polynomial}
\begin{split}
B_{n}(x)&=\sum_{k=0}^{n}{n\choose k}B_{k}x^{n-k}.
\end{split}
\end{equation}
It holds for all $0\leq y\leq1$ that \cite[\text{Corollary B.4, (B.21), p.\;500}]{13}
\begin{equation}
\left|B_{1}(y)\right|\leq\frac{1}{2}\;\;\text{and that}\;\;\left|B_{n}(y)\right|\leq\frac{2\zeta(n)n!}{(2\pi)^{n}}\;\;\text{for all $n\in\mathbb{N}_{\geq2}$}.
\end{equation}
We have that \cite[\text{(1.10), p.\;282}]{14}
\begin{equation}\label{Bernoulli relation}
(-1)^{k}B_{k}(1-y)=B_{k}(y)\;\;\text{for all $k\in\mathbb{N}_{0}$ and $0\leq y\leq1$}.
\end{equation}
\end{definition}

\begin{definition}\label{digammafunction}(Digamma function)\;\cite[\text{pp.\;136-138}]{9}\\
We set the \emph{digamma function} $\psi(z)$ to
\begin{equation}
\begin{split}
\psi(z):&=\frac{\Gamma'(z)}{\Gamma(z)}\;\;\text{for all $z\in\mathbb{C}\setminus\{0,-1,-2,-3,\ldots\}$}.
\end{split}
\end{equation}
Therefore, $\psi(z)$ is an analytic function for all $z\in\mathbb{C}\setminus{(-\infty,0]}$. For all $z\in\mathbb{C}\setminus\{0,-1,-2,-3,\ldots\}$, we have the identity \cite[\text{(5.5.2), p.\;138}]{9}
\begin{equation}\label{digamma identity}
\begin{split}
\psi(z+1)&=\psi(z)+\frac{1}{z}
\end{split}
\end{equation}
and for all $n\in\mathbb{N}$ we have the formula \cite[\text{(5.4.14), p.\;137}]{9}
\begin{equation}\label{digamma summation formula}
\begin{split}
\sum_{k=1}^{n}\frac{1}{k}&=\psi(n+1)+\gamma.
\end{split}
\end{equation}
\end{definition}

\begin{definition}\label{Hurwitzzetafunction}(Hurwitz zeta function)\;\cite[\text{p.\;607}]{9}\\
We define the \emph{Hurwitz zeta function} $\zeta(s,z)$ for all complex numbers $s\in\mathbb{C}$ with $\re(s)>1$ and all $z\in\mathbb{C}\setminus\{0,-1,-2,-3,\ldots\}$ by
\begin{equation}
\begin{split}
\zeta(s,z):&=\sum_{k=0}^{\infty}\frac{1}{(k+z)^{s}}.
\end{split}
\end{equation}
The function $\zeta(s,z)$ extends to an analytic function on $\mathbb{C}\setminus\{0,-1,-2,-3,\ldots\}$ in the $z$-variable and for every $z\notin\{0,-1,-2,-3,\ldots\}$ to a meromorphic function in $s\in\mathbb{C}\setminus\{1\}$ with a simple pole at $s=1$. It satisfies for all $s\in\mathbb{C}\setminus\{1\}$ and all $z\in\mathbb{C}\setminus\{0,-1,-2,-3,\ldots\}$ the identity
\begin{equation}\label{Hurwitz zeta identity}
\begin{split}
\zeta(s,z+1)&=\zeta(s,z)-\frac{1}{z^{s}}.
\end{split}
\end{equation}
For all $m\in\mathbb{C}\setminus\{1\}$ and all $n\in\mathbb{N}$ we have the formula \cite[\text{(1.2), p.\;2}]{15}
\begin{equation}\label{Hurwitz summation formula}
\begin{split}
\sum_{k=1}^{n}k^{m}&=\zeta(-m)-\zeta(-m,n+1).
\end{split}
\end{equation}
\end{definition}

\section{The Structure of Inverse Factorial Series Expansions}
\label{sec:The Structure of Inverse Factorial Series Expansions}

In this section we study the structure of inverse factorial series expansions for analytic functions possessing an asymptotic series expansion by applying a Theorem of G. N. Watson \cite[\text{Theorem\;2, p.\;45}]{16}. The main result of this section is Theorem \ref{Structure of inverse factorial series expansions}, from which we will later deduce convergent inverse factorial series expansions for the functions $\zeta(s,z+1-y)$ and $\psi(z+1-y)$, where $0\leq y\leq1$.\\
For this procedure, we need the following variant of a result found by J. Weniger \cite[\text{(4.1), p.\;1433}]{2}

\begin{lemma}(finite Weniger transformation)\label{finite Weniger transformation}\;\cite{2}\\
For every finite inverse power series $\sum_{k=1}^{n}\frac{a_{k}}{z^{k}}$, where the $a_{k}$'s are any complex numbers and $n\in\mathbb{N}$, the following transformation formula holds
\begin{equation}
\label{finite Weniger transformation formula}
\begin{split}
\sum_{k=1}^{n}\frac{a_{k}}{z^{k}}&=\sum_{k=1}^{n}\frac{(-1)^{k}\sum_{l=1}^{k}(-1)^{l}S_{k}^{(1)}(l)a_{l}}{(z+1)(z+2)\cdots(z+k)}+O\left(\frac{1}{z^{n+1}}\right)\;\;\;\;\text{as $z\rightarrow\infty$}.
\end{split}
\end{equation}
Moreover, we have that
\begin{equation}
\label{Weniger transformation formula}
\begin{split}
\sum_{k=1}^{n}\frac{a_{k}}{z^{k}}&=\sum_{k=1}^{\infty}\frac{(-1)^{k}\sum_{l=1}^{n}(-1)^{l}S_{k}^{(1)}(l)a_{l}}{(z+1)(z+2)\cdots(z+k)}.
\end{split}
\end{equation}
\end{lemma}

\begin{proof}
We have for $l\in\mathbb{N}$ that \cite[\text{(A.14), p.\;1438}]{2}, \cite[\text{(6), p.\;78}]{17}
\begin{displaymath}
\begin{split}
\frac{1}{z^{l}}&=\sum_{k=0}^{\infty}\frac{(-1)^{k}S_{k+l}^{(1)}(l)}{(z+1)(z+2)\cdots(z+k+l)}\\
&=\sum_{k=1}^{\infty}\frac{(-1)^{k-l}S_{k}^{(1)}(l)}{(z+1)(z+2)\cdots(z+k)}\\
&=\sum_{k=1}^{n}\frac{(-1)^{k-l}S_{k}^{(1)}(l)}{(z+1)(z+2)\cdots(z+k)}+O\left(\frac{1}{z^{n+1}}\right)\;\;\text{as $z\rightarrow\infty$}.
\end{split}
\end{displaymath}
Therefore, we obtain that
\begin{displaymath}
\begin{split}
\sum_{k=1}^{n}\frac{a_{k}}{z^{k}}=\sum_{l=1}^{n}\frac{a_{l}}{z^{l}}&=\sum_{l=1}^{n}a_{l}\sum_{k=1}^{n}\frac{(-1)^{k-l}S_{k}^{(1)}(l)}{(z+1)(z+2)\cdots(z+k)}+O\left(\frac{1}{z^{n+1}}\right)\\
&=\sum_{k=1}^{n}\frac{(-1)^{k}\sum_{l=1}^{n}(-1)^{l}S_{k}^{(1)}(l)a_{l}}{(z+1)(z+2)\cdots(z+k)}+O\left(\frac{1}{z^{n+1}}\right)\\
&=\sum_{k=1}^{n}\frac{(-1)^{k}\sum_{l=1}^{k}(-1)^{l}S_{k}^{(1)}(l)a_{l}}{(z+1)(z+2)\cdots(z+k)}+O\left(\frac{1}{z^{n+1}}\right)\;\;\text{as $z\rightarrow\infty$},
\end{split}
\end{displaymath}
which is the first claimed formula \eqref{finite Weniger transformation formula}.\\
The second formula \eqref{Weniger transformation formula} follows from the calculation
\begin{displaymath}
\begin{split}
\sum_{k=1}^{n}\frac{a_{k}}{z^{k}}=\sum_{l=1}^{n}\frac{a_{l}}{z^{l}}&=\sum_{l=1}^{n}a_{l}\sum_{k=1}^{\infty}\frac{(-1)^{k-l}S_{k}^{(1)}(l)}{(z+1)(z+2)\cdots(z+k)}\\
&=\sum_{k=1}^{\infty}\frac{(-1)^{k}\sum_{l=1}^{n}(-1)^{l}S_{k}^{(1)}(l)a_{l}}{(z+1)(z+2)\cdots(z+k)},
\end{split}
\end{displaymath}
because we can always interchange a finite summation with an infinite summation.
\end{proof}

\begin{lemma}\label{Uniqueness of inverse factorial series expansions}(Uniqueness of inverse factorial series expansions)\\
If a function $f(z)$ has for all $z\in\mathbb{C}$ with $\re(z)>0$ the absolutely convergent series expansion
\begin{displaymath}
\begin{split}
f(z)&=\sum_{k=1}^{\infty}\frac{b_{k}}{(z+1)(z+2)\cdots(z+k)}
\end{split}
\end{displaymath}
and the asymptotic expansion
\begin{displaymath}
\begin{split}
f(z)&=\sum_{k=1}^{n}\frac{c_{k}}{(z+1)(z+2)\cdots(z+k)}+O\left(\frac{1}{z^{n+1}}\right)\;\;\;\;\text{as $z\rightarrow\infty$},
\end{split}
\end{displaymath}
then we have that $c_{k}=b_{k}$ for all $k\in\mathbb{N}$ and we have the absolutely convergent series expansion
\begin{displaymath}
\begin{split}
f(z)&=\sum_{k=1}^{\infty}\frac{c_{k}}{(z+1)(z+2)\cdots(z+k)}.
\end{split}
\end{displaymath}
\end{lemma}

\begin{proof}
From the given absolutely convergent inverse factorial series expansion of $f(z)$, we deduce for all $n\in\mathbb{N}$ that
\begin{displaymath}
\begin{split}
\sum_{k=n}^{\infty}\frac{b_{k}}{(z+1)(z+2)\cdots(z+k)}&\leq\sum_{k=n}^{\infty}\frac{|b_{k}|}{|(z+1)||(z+2)|\cdots|(z+k)|}=O\left(\frac{1}{z^{n}}\right)\;\;\text{as $z\rightarrow\infty$},
\end{split}
\end{displaymath}
which means that $\lim_{z\rightarrow\infty}(f(z))=0$ and that we have
\begin{displaymath}
\begin{split}
z^{m}\sum_{k=n}^{\infty}\frac{b_{k}}{(z+1)(z+2)\cdots(z+k)}\longrightarrow0\;\;\text{as $z\rightarrow\infty$ for all $m\in\{0,1,2,\ldots,n-1\}$}.
\end{split}
\end{displaymath}
The result now follows by induction on $n\in\mathbb{N}$ via a repeated application of the above limit.\\
For $n=1$, we get
\begin{displaymath}
\begin{split}
f(z)&=\frac{b_{1}}{z+1}+\sum_{k=2}^{\infty}\frac{b_{k}}{(z+1)(z+2)\cdots(z+k)}
=\frac{c_{1}}{z+1}+O\left(\frac{1}{z^{2}}\right)\;\;\;\;\text{as $z\rightarrow\infty$},
\end{split}
\end{displaymath}
which implies by multiplying both sides with $z+1$ and letting $z\rightarrow\infty$ that $c_{1}=b_{1}$.\\
Similarly, for $n=2$, we get using $c_{1}=b_{1}$ that
\begin{displaymath}
\begin{split}
f(z)-\frac{c_{1}}{z+1}&=\frac{b_{2}}{(z+1)(z+2)}+\sum_{k=3}^{\infty}\frac{b_{k}}{(z+1)(z+2)\cdots(z+k)}
=\frac{c_{2}}{(z+1)(z+2)}+O\left(\frac{1}{z^{3}}\right),
\end{split}
\end{displaymath}
which implies by multiplying both sides with $(z+1)(z+2)$ and letting $z\rightarrow\infty$ that $c_{2}=b_{2}$.\\
In general, we can induct from $n-1$ to $n$ using that $c_{k}=b_{k}$ for all $k\in\{1,2,3,\ldots,n-1\}$ by the identity
\begin{displaymath}
\begin{split}
f(z)-\sum_{k=1}^{n-1}\frac{c_{k}}{(z+1)(z+2)\cdots(z+k)}&=\frac{b_{n}}{(z+1)(z+2)\cdots(z+n)}+\sum_{k=n+1}^{\infty}\frac{b_{k}}{(z+1)(z+2)\cdots(z+k)}\\
&=\frac{c_{n}}{(z+1)(z+2)\cdots(z+n)}+O\left(\frac{1}{z^{n+1}}\right)\;\;\;\;\text{as $z\rightarrow\infty$},
\end{split}
\end{displaymath}
again by multiplying both sides with $(z+1)(z+2)\cdots(z+n)$ and letting $z\rightarrow\infty$ to conclude that $c_{k}=b_{k}$ holds also for $k=n$.\\
This proves that $c_{k}=b_{k}$ for all $k\in\mathbb{N}$.
\end{proof}

\noindent The key to our generalized Faulhaber formulas will be the following
\begin{theorem}\label{Watson's Transformation Theorem}(Watson's Transformation Theorem)\;\cite[\text{Theorem\;2, p.\;45}]{16}\\
Let $f(z)$ be a function of $z\in\mathbb{C}$ which is analytic when $\re(z)>0$; and let $f(z)$ be also analytic in the region $D$ of the complex plane defined by
\begin{displaymath}
\begin{split}
D:&=\left\{z\in\mathbb{C}:|z|>\gamma\;\text{and}\;|\arg(z)|\leq\frac{\pi}{2}+\alpha+3\delta\right\},
\end{split}
\end{displaymath}
where $\gamma\geq0$ is a finite number, $\alpha>0$, $\delta>0$ and $\alpha+3\delta<\frac{\pi}{2}$.\\
In the region $D$ let $f(z)$ possess the asymptotic expansion
\begin{displaymath}
\begin{split}
f(z)&=\sum_{k=0}^{n}\frac{a_{k}}{z^{k}}+R_{n}(z)=a_{0}+\frac{a_{1}}{z}+\frac{a_{2}}{z^{2}}+\frac{a_{3}}{z^{3}}+\frac{a_{4}}{z^{4}}+\ldots+\frac{a_{n}}{z^{n}}+R_{n}(z),
\end{split}
\end{displaymath}
where
\begin{displaymath}
\begin{split}
|a_{n}|&<A\rho^{n}n!\;\;\;\;\text{and}\;\;\;\;|R_{n}(z)z^{n+1}|<B\sigma^{n}n!,
\end{split}
\end{displaymath}
with some constants $A$, $B$, $\rho$ and $\sigma$, which are independent of $n$.\\
Let $M\leq M_{0}$ be any positive real number, where $M_{0}$ is the largest positive root of the equation
\begin{displaymath}
\begin{split}
e^{-\frac{2\cos(\alpha)}{\rho M_{0}}}-2\cos\left(\frac{\sin(\alpha)}{\rho M_{0}}\right)\cdot e^{-\frac{\cos(\alpha)}{\rho M_{0}}}+1-p^{2}&=0,
\end{split}
\end{displaymath}
where
\begin{displaymath}
\begin{split}
1&<p<1+e^{-\pi\cot(\alpha)}.
\end{split}
\end{displaymath}
Then the function $f(z)$ can be expanded into the absolutely convergent series
\begin{displaymath}
\begin{split}
f(z)&=b_{0}+\sum_{k=1}^{\infty}\frac{b_{k}}{(Mz+w+1)(Mz+w+2)\cdots(Mz+w+k)},
\end{split}
\end{displaymath}
when $\re(z)>0$ and $w\in\mathbb{C}$ with $\re(w)\geq0$.
\end{theorem}

\begin{proof}
The proof of this Theorem is given in Watson's paper \cite{16}.
\end{proof}

It follows
\begin{theorem}\label{Structure of inverse factorial series expansions}(Structure of inverse factorial series expansions)\\
Let $f(z)$ be a function of $z\in\mathbb{C}$ which is analytic when $\re(z)>0$; and let $f(z)$ be also analytic in the region $D$ of the complex plane defined by
\begin{displaymath}
\begin{split}
D:&=\left\{z\in\mathbb{C}:|z|>0\;\text{and}\;|\arg(z)|\leq\pi-\varepsilon,\;\text{where $\varepsilon>0$ is arbitrarily small}\right\}.
\end{split}
\end{displaymath}
In the region $D$ let $f(z)$ possess the asymptotic expansion
\begin{displaymath}
\begin{split}
f(z)&=\sum_{k=1}^{n}\frac{a_{k}}{z^{k}}+R_{n}(z)=\frac{a_{1}}{z}+\frac{a_{2}}{z^{2}}+\frac{a_{3}}{z^{3}}+\ldots+\frac{a_{n}}{z^{n}}+R_{n}(z),
\end{split}
\end{displaymath}
where
\begin{displaymath}
\begin{split}
|a_{n}|&<A\rho^{n}n!\;\;\;\;\text{and}\;\;\;\;|R_{n}(z)z^{n+1}|<B\sigma^{n}n!,
\end{split}
\end{displaymath}
with some constants $A$, $B$, $\rho<\frac{3}{\pi}$ and $\sigma$, which are independent of $n$.\\
Then the function $f(z)$ is equal to the absolutely convergent inverse factorial series
\begin{equation}\label{structure of inverse factorial series expansion}
\begin{split}
f(z)&=\sum_{k=1}^{\infty}\frac{(-1)^{k}\sum_{l=1}^{k}(-1)^{l}S_{k}^{(1)}(l)a_{l}}{(z+1)(z+2)\cdots(z+k)}
\end{split}
\end{equation}
for $\re(z)>0$.
\end{theorem}

\begin{proof}
Let $f(z)$ and the region $D$ be as described in the above Theorem \ref{Structure of inverse factorial series expansions}. Because of the conditions on the function $f(z)$ and the region $D$, we can choose in Theorem \ref{Watson's Transformation Theorem} the variables $\gamma:=0$, $\alpha:=\frac{\pi}{2}-4\varepsilon$, $\delta:=\varepsilon$ and $p:=1+\varepsilon$ for $\varepsilon>0$ arbitrarily small by \cite[\text{beginning of p.\;85}]{16}. We have then that $M_{0}=\frac{3}{\pi\rho}-\varepsilon$ for some arbitrarily small number $\varepsilon>0$ and because $\rho<\frac{3}{\pi}$, we obtain that $M_{0}>1$. According to Watson's Theorem \ref{Watson's Transformation Theorem} with $M:=1<M_{0}$, $w:=0$ and $a_{0}=b_{0}=0$, we know that we can expand the function $f(z)$ into an absolutely convergent series of the form
\begin{displaymath}
\begin{split}
f(z)&=\sum_{k=1}^{\infty}\frac{b_{k}}{(z+1)(z+2)\cdots(z+k)}
\end{split}
\end{displaymath}
for some constants $b_{k}\in\mathbb{C}$ and all $z\in\mathbb{C}$ with $\re(z)>0$.\\
On the other hand, we have by applying a finite Weniger transformation \eqref{finite Weniger transformation formula} to the asymptotic expansion of $f(z)$ that\begin{displaymath}
\begin{split}
f(z)&=\sum_{k=1}^{n}\frac{(-1)^{k}\sum_{l=1}^{k}(-1)^{l}S_{k}^{(1)}(l)a_{l}}{(z+1)(z+2)\cdots(z+k)}+O\left(\frac{1}{z^{n+1}}\right)\;\;\text{as $z\rightarrow\infty$}
\end{split}
\end{displaymath}
also holds.\\
Comparing the two above expressions for $f(z)$ by using Lemma \ref{Uniqueness of inverse factorial series expansions} with $c_{k}:=(-1)^{k}\sum_{l=1}^{k}(-1)^{l}S_{k}^{(1)}(l)a_{l}$, we conclude that we must have the absolutely convergent series
\begin{displaymath}
\begin{split}
f(z)&=\sum_{k=1}^{\infty}\frac{(-1)^{k}\sum_{l=1}^{k}(-1)^{l}S_{k}^{(1)}(l)a_{l}}{(z+1)(z+2)\cdots(z+k)}
\end{split}
\end{displaymath}
for all $z\in\mathbb{C}$ with $\re(z)>0$.
\end{proof}

\section{The Convergent Inverse Factorial Series Expansions for $\zeta(s,z+1-y)$ and $\psi(z+1-y)$}
\label{sec:The Series Expansions for zeta(s,z+1-y) and psi(z+1-y)}

In this section, we deduce in Theorem \ref{Inverse factorial series expansions for zeta(s,z+1-y) and psi(z+1-y)} the convergent inverse factorial series expansions for the functions $\zeta(s,z+1-y)$ and $\psi(z+1-y)$, where $0\leq y\leq1$.\\
For this, we need the following
\begin{lemma}(Euler-Maclaurin summation formula)\cite[\text{Theorem B.5, pp.\;500-501}]{13}\label{Euler-Maclaurin formula}\\
Suppose that $n\in\mathbb{N}$ is a positive integer and that the function $f(t)$ has continuous derivatives through the $n$-th order on the interval $[a,b]$ where $a$ and $b$ are real numbers with $a<b$.\\
Then we have
\begin{equation}
\label{Euler-Maclaurin summation formula}
\begin{split}
\sum_{a<k\leq b}f(k)&=\int_{a}^{b}f(t)dt+\sum_{k=1}^{n}(-1)^{k}\frac{B_{k}(\{b\})}{k!}f^{(k-1)}(b)-\sum_{k=1}^{n}(-1)^{k}\frac{B_{k}(\{a\})}{k!}f^{(k-1)}(a)\\
&\quad+\frac{(-1)^{n+1}}{n!}\int_{a}^{b}f^{(n)}(t)B_{n}(\{t\})dt.
\end{split}
\end{equation}
\end{lemma}

\begin{proof}
The proof of this Lemma \ref{Euler-Maclaurin formula} is given in \cite[\text{p.\;501}]{13}.
\end{proof}

From the above Lemma \ref{Euler-Maclaurin formula}, it follows
\begin{lemma}\label{Asymptotic series expansions for zeta(s,z+h) and psi(z+h) with 0leq hleq1}(Asymptotic series expansions for $\zeta(s,z+h)$ and $\psi(z+h)$ with $0\leq h\leq1$)\\
Let $n\in\mathbb{N}_{0}$ and let $s\in\mathbb{C}\setminus\{1\}$ such that $\re(s)>-n$. We have for $z\in\mathbb{C} $ with $|\arg(z)|<\pi$ and $0\leq h\leq1$ the asymptotic series expansions
\begin{equation}\label{Hurwitz zeta expansion}
\begin{split}
\zeta(s,z+h)&=\frac{z^{1-s}}{s-1}+\frac{z^{1-s}}{s-1}\sum_{k=1}^{n}{1-s\choose k}\frac{B_{k}(h)}{z^{k}}+O_{n}(z)
\end{split}
\end{equation}
with
\begin{equation}\label{error term O_n(z)}
\begin{split}
\left|O_{n}(z)\right|&=\left|{1-s\choose n+2}\frac{n+2}{s-1}\int_{0}^{\infty}\frac{B_{n+1}\left(\{x-h\}\right)-(-1)^{n+1}B_{n+1}(h)}{(x+z)^{n+s+1}}dx\right|\\
&\leq\frac{2(n+2)}{|s-1|}\left|{1-s\choose n+2}\right|\frac{\left|B_{n+1}\right|\sec^{n+\re(s)+1}\left(\frac{1}{2}\arg(z)\right)}{(n+\re(s))|z|^{n+\re(s)}}\max\left\{1,e^{\im(s)\arg(z)}\right\}
\end{split}
\end{equation}
and
\begin{equation}\label{digamma expansion}
\begin{split}
\psi(z+h)&=\log(z)-\sum_{k=1}^{n}\frac{(-1)^{k}B_{k}(h)}{kz^{k}}+U_{n}(z)
\end{split}
\end{equation}
with
\begin{equation}\label{error term U_n(z)}
\begin{split}
\left|U_{n}(z)\right|&=\left|\int_{0}^{\infty}\frac{(-1)^{n+1}B_{n+1}(h)-B_{n+1}(\{x-h\})}{(x+z)^{n+2}}dx\right|\leq\frac{2\left|B_{n+1}\right|\sec^{n+2}\left(\frac{1}{2}\arg(z)\right)}{(n+1)|z|^{n+1}}.
\end{split}
\end{equation}
\end{lemma}

\begin{proof}
Let $0<h\leq1$ and let $z\in\mathbb{C}$ with $\left|\arg(z)\right|<\pi$. Setting $a:=-h$, $b:=N$ and $f(x):=\frac{1}{(z+h+x)^{s}}$ with $\frac{d^{n}f(x)}{dx^{n}}=\frac{d^{n}}{dx^{n}}\left(\frac{1}{(z+h+x)^{s}}\right)=-\frac{(n+1)!}{s-1}{1-s\choose n+1}\frac{1}{(z+h+x)^{n+s}}$ into Lemma \ref{Euler-Maclaurin formula}, we obtain for $s\in\mathbb{C}$ with $\re(s)>1$ that
\begin{displaymath}
\begin{split}
\zeta(s,z+h)&=\sum_{k=0}^{\infty}\frac{1}{(z+h+k)^{s}}=\lim_{N\rightarrow\infty}\left(\sum_{-h<k\leq N}\frac{1}{(z+h+k)^{s}}\right)\\
&=\int_{-h}^{\infty}\frac{1}{(z+h+x)^{s}}dx+\lim_{N\rightarrow\infty}\left(\sum_{k=1}^{n}(-1)^{k}\frac{B_{k}(\{N\})}{k!}\frac{d^{k-1}\left(\frac{1}{(z+h+x)^{s}}\right)}{dx^{k-1}}\Bigg|_{x=N}\right)\\
&\quad-\sum_{k=1}^{n}(-1)^{k}\frac{B_{k}(\{-h\})}{k!}\frac{d^{k-1}\left(\frac{1}{(z+h+x)^{s}}\right)}{dx^{k-1}}\Bigg|_{x=-h}+\frac{(-1)^{n+1}}{n!}\int_{-h}^{\infty}\frac{d^{n}\left(\frac{1}{(z+h+x)^{s}}\right)}{dx^{n}}\Bigg|_{x=t}B_{n}(\{t\})dt\\
&=\frac{z^{1-s}}{s-1}+\frac{z^{1-s}}{s-1}\sum_{k=1}^{n}(-1)^{k}{1-s\choose k}\frac{B_{k}(\{1-h\})}{z^{k}}+(-1)^{n}{1-s\choose n+1}\frac{n+1}{s-1}\int_{0}^{\infty}\frac{B_{n}(\{x-h\})}{(x+z)^{n+s}}dx.
\end{split}
\end{displaymath}
This is equivalent to
\begin{displaymath}
\begin{split}
\zeta(s,z+h)&=\frac{z^{1-s}}{s-1}+\frac{z^{1-s}}{s-1}\sum_{k=1}^{n}(-1)^{k}{1-s\choose k}\frac{B_{k}(1-h)}{z^{k}}\\
&\quad+(-1)^{n+1}{1-s\choose n+2}\frac{n+2}{s-1}\int_{0}^{\infty}\frac{B_{n+1}(\{x-h\})-B_{n+1}(1-h)}{(x+z)^{n+s+1}}dx.
\end{split}
\end{displaymath}
We use the relation \eqref{Bernoulli relation} and deduce equation \eqref{Hurwitz zeta expansion}, which extends $\zeta(s,z+h)$ analytically to the whole punctured complex $s$-plane $\mathbb{C}\setminus\{1\}$. Therefore, the equation \eqref{Hurwitz zeta expansion} is also true for all $s\in\mathbb{C}\setminus\{1\}$. By using the identity \eqref{Hurwitz zeta identity}, we see that the formula \eqref{Hurwitz zeta expansion} is also true for $h=0$. The bound \eqref{error term O_n(z)} for the error term $O_{n}(z)$ follows from \cite[\text{p.\;294}]{14} and \cite[\text{p.\;6}]{15}. This proves the first part about $\zeta(s,z+h)$ of the above Lemma \ref{Asymptotic series expansions for zeta(s,z+h) and psi(z+h) with 0leq hleq1}.\\
Now, we prove the second part for $\psi(z+h)$. We have for $z\in\mathbb{C}$ with $\left|\arg(z)\right|<\pi$, $n\geq2$ and $0\leq h\leq1$ the series expansion \cite[\text{Ex.\;4.4, p.\;295}]{14}
\begin{equation}\label{LogGammaExpansion}
\begin{split}
\ln\left(\Gamma(z+h)\right)&=\left(z+h-\frac{1}{2}\right)\ln(z)-z+\frac{1}{2}\ln(2\pi)+\sum_{k=2}^{n}\frac{(-1)^{k}B_{k}(h)}{k(k-1)z^{k-1}}-\frac{1}{n}\int_{0}^{\infty}\frac{B_{n}\left(\{x-h\}\right)}{(x+z)^{n}}dx\\
&\hspace{-1.8cm}=\left(z+h-\frac{1}{2}\right)\ln(z)-z+\frac{1}{2}\ln(2\pi)+\sum_{k=2}^{n}\frac{(-1)^{k}B_{k}(h)}{k(k-1)z^{k-1}}+\frac{1}{n+1}\int_{0}^{\infty}\tfrac{(-1)^{n+1}B_{n+1}(h)-B_{n+1}\left(\{x-h\}\right)}{(x+z)^{n+1}}dx.
\end{split}
\end{equation}
Differentiating this identity with respect to the variable $z$, we get equation \eqref{digamma expansion}. The estimate \eqref{error term U_n(z)} for the error term $U_{n}(z)$ follows from \cite[\text{p.\;294 and Ex.\;4.2, p.\;295}]{14}.
\end{proof}

We get the following
\begin{theorem}\label{Inverse factorial series expansions for zeta(s,z+1-y) and psi(z+1-y)}(Inverse factorial series expansions for $\zeta(s,z+1-y)$ and $\psi(z+1-y)$)\\
Let $0\leq y\leq1$ and $s\in\mathbb{C}\setminus\{1\}$. We have for $z\in\mathbb{H}^{+}$ and all $a\in\mathbb{N}_{0}$ the absolutely convergent inverse factorial series expansions
\begin{equation}\label{inverse factorial series expansion for the Hurwitz zeta function}
\begin{split}
\zeta(s,z+1-y)&=\frac{z^{1-s}}{s-1}+\frac{z^{1-s}}{s-1}\sum_{k=1}^{a}(-1)^{k}{1-s\choose k}\frac{B_{k}(y)}{z^{k}}\\
&\quad+\frac{z^{1-s-a}}{s-1}\sum_{k=1}^{\infty}(-1)^{k+a}\frac{\sum_{l=1}^{k}{1-s\choose l+a}S_{k}^{(1)}(l)B_{l+a}(y)}{(z+1)(z+2)\cdots(z+k)}
\end{split}
\end{equation}
and
\begin{equation}\label{inverse factorial series expansion for the digamma function}
\begin{split}
\psi(z+1-y)&=\log(z)-\sum_{k=1}^{a}\frac{B_{k}(y)}{kz^{k}}+\frac{1}{z^{a}}\sum_{k=1}^{\infty}(-1)^{k+1}\frac{\sum_{l=1}^{k}\frac{(-1)^{l}}{l+a}S_{k}^{(1)}(l)B_{l+a}(y)}{(z+1)(z+2)\cdots(z+k)}.
\end{split}
\end{equation}
\end{theorem}

\begin{proof}
Let $s\in\mathbb{C}\setminus\{1,0,-1,-2,-3,\ldots\}$ be a fixed complex number and let $z\in\mathbb{C}\setminus(-\infty,0]$ with $|\arg(z)|\leq\pi-\varepsilon<\pi$ for some arbitrarily small, but fixed $\varepsilon>0$. Setting $h:=1-y$ for $0\leq y\leq1$ into the identities \eqref{Hurwitz zeta expansion} and \eqref{digamma expansion}, we deduce by using the relation \eqref{Bernoulli relation} and by exchanging $n$ with $n+a$ that
\begin{equation}\label{formula1}
\begin{split}
\zeta(s,z+1-y)&=\frac{z^{1-s}}{s-1}+\frac{z^{1-s}}{s-1}\sum_{k=1}^{n+a}(-1)^{k}{1-s\choose k}\frac{B_{k}(y)}{z^{k}}+O_{n+a}(z)
\end{split}
\end{equation}
and that
\begin{equation}\label{formula2}
\begin{split}
\psi(z+1-y)&=\log(z)-\sum_{k=1}^{n+a}\frac{B_{k}(y)}{kz^{k}}+U_{n+a}(z),
\end{split}
\end{equation}
where $O_{n+a}(z)$ and $U_{n+a}(z)$ are as in the previous Lemma \ref{Asymptotic series expansions for zeta(s,z+h) and psi(z+h) with 0leq hleq1} with $h:=1-y$.\\
We can write the equations \eqref{formula1} and \eqref{formula2} in the forms
\begin{equation}\label{formula3}
\begin{split}
\zeta(s,z+1-y)&=\frac{z^{1-s}}{s-1}+\frac{z^{1-s}}{s-1}\sum_{k=1}^{a}(-1)^{k}{1-s\choose k}\frac{B_{k}(y)}{z^{k}}\\
&\quad+\frac{z^{1-s-a}}{s-1}\sum_{k=1}^{n}(-1)^{k+a}{1-s\choose k+a}\frac{B_{k+a}(y)}{z^{k}}+O_{n+a}(z)
\end{split}
\end{equation}
and
\begin{equation}\label{formula4}
\begin{split}
\psi(z+1-y)&=\log(z)-\sum_{k=1}^{a}\frac{B_{k}(y)}{kz^{k}}-\frac{1}{z^{a}}\sum_{k=1}^{n}\frac{B_{k+a}(y)}{(k+a)z^{k}}+U_{n+a}(z).
\end{split}
\end{equation}
In the following calculations, we will use that the function $g(k):=2^{k}$ grows faster than any polynomial $p(k)$ as $k\rightarrow\infty$.\\
From equation \eqref{formula3}, we get \eqref{inverse factorial series expansion for the Hurwitz zeta function} for $s\in\mathbb{C}\setminus\{1,0,-1,-2,-3,\ldots\}$ by applying Theorem \ref{Structure of inverse factorial series expansions} with $R_{n}(z):=(s-1)z^{s+a-1}O_{n+a}(z)$ to the analytic function $f_{1}(z)$ defined by
\begin{displaymath}
\begin{split}
f_{1}(z):&=(s-1)z^{s+a-1}\left[\zeta(s,z+1-y)-\frac{z^{1-s}}{s-1}-\frac{z^{1-s}}{s-1}\sum_{k=1}^{a}(-1)^{k}{1-s\choose k}\frac{B_{k}(y)}{z^{k}}\right]\\
&=\sum_{k=1}^{n}(-1)^{k+a}{1-s\choose k+a}\frac{B_{k+a}(y)}{z^{k}}+(s-1)z^{s+a-1}O_{n+a}(z)
\end{split}
\end{displaymath}
on $z\in\mathbb{C}\setminus(-\infty,0]$ with $|\arg(z)|\leq\pi-\varepsilon<\pi$, because using that $|x^{s}|=x^{\re(s)}$ for $x\in\mathbb{R}_{0}^{+}$, we have by employing the identity \eqref{binomial coefficient bound} that
\begin{displaymath}
\begin{split}
\left|{1-s\choose k+a}B_{k+a}(y)\right|&\leq\frac{2\zeta(k+a)(k+a)!(k+a)^{\re(s)-2}}{(2\pi)^{k+a}\left|\Gamma(s-1)\right|}+O\left(\frac{2\zeta(k+a)(k+a)!(k+a)^{\re(s)-3}}{(2\pi)^{k+a}}\right)\\
&<\frac{C_{1}(a)k!}{\pi^{k}}
\end{split}
\end{displaymath}
and with $A_{1}:=\max\left\{1,e^{\im(s)\arg(z)}\right\}$, $A_{2}:=\max\left\{1,e^{-\im(s)\arg(z)}\right\}$, as well as $\re(s)>-n$ that
\begin{displaymath}
\begin{split}
\left|(s-1)z^{s+a-1}O_{n+a}(z)\right|&\leq2(n+a+2)A_{1}\left|{1-s\choose n+a+2}\right|\frac{\left|B_{n+a+1}\right|e^{-\im(s)\arg(z)}\sec^{n+a+\re(s)+1}\left(\frac{1}{2}\arg(z)\right)}{\left|n+a+\re(s)\right|\cdot|z|^{n+1}}\\
&\hspace{-3.6cm}\leq\frac{4(n+a+2)A_{2}}{\left|n+a+\re(s)\right|}\cdot\frac{\zeta(n+a+1)(n+a+1)!(n+a+2)^{\re(s)-2}\sec^{n+a+\re(s)+1}\left(\frac{1}{2}\arg(z)\right)}{(2\pi)^{n+a+1}\left|\Gamma(s-1)\right||z|^{n+1}}\\
&\hspace{-3.6cm}\quad+O\left(\frac{4(n+a+2)A_{2}}{\left|n+a+\re(s)\right|}\cdot\frac{\zeta(n+a+1)(n+a+1)!(n+a+2)^{\re(s)-3}\sec^{n+a+\re(s)+1}\left(\frac{1}{2}\arg(z)\right)}{(2\pi)^{n+a+1}|z|^{n+1}}\right)\\
&\hspace{-3.6cm}<\frac{C_{2}(a)\sec^{n}\left(\frac{1}{2}\arg(z)\right)n!}{\pi^{n}|z|^{n+1}}
\end{split}
\end{displaymath}
for some positive constants $C_{1}(a)$, $C_{2}(a)$ depending on $a$ and independent of $n$. In the last computation above, we have used the relation $|z^{s}|=|z|^{\re(s)}e^{-\im(s)\arg(z)}$.\\
The above bound for $\left|(s-1)z^{s+a-1}O_{n+a}(z)\right|$ is also true if $\re(s)\leq-n$ by taking $C_{2}(a)$ large enough, because $\re(s)\leq-n$ is only possible for finitely many $n$'s and in each case we have that $\left|(s-1)z^{s+a-1}O_{n+a}(z)\right|\leq\frac{C(n)}{|z|^{n+1}}$ for all $n\in\mathbb{N}$ and some positive constants $C(n)$.\\
To get formula \eqref{inverse factorial series expansion for the Hurwitz zeta function} also for all $s\in\{0,-1,-2,-3,\ldots\}$, we apply the Weniger transformation formula \eqref{Weniger transformation formula} directly to the function $f_{1}(z)$ with $n:=1-s-a$ and $O_{n+a}(z)=O_{1-s}(z)=0$.\\
Similarly from equation \eqref{formula4}, we obtain the formula \eqref{inverse factorial series expansion for the digamma function} by applying Theorem \ref{Structure of inverse factorial series expansions} with $R_{n}(z):=z^{a}U_{n+a}(z)$ to the analytic function $f_{2}(z)$ defined by
\begin{displaymath}
\begin{split}
f_{2}(z):&=z^{a}\left[\log(z)-\psi(z+1-y)-\sum_{k=1}^{a}\frac{B_{k}(y)}{kz^{k}}\right]=\sum_{k=1}^{n}\frac{B_{k+a}(y)}{(k+a)z^{k}}+z^{a}U_{n+a}(z)
\end{split}
\end{displaymath}
on $z\in\mathbb{C}$ with $|\arg(z)|\leq\pi-\varepsilon<\pi$, because we have
\begin{displaymath}
\begin{split}
\left|\frac{B_{k+a}(y)}{k+a}\right|\leq\frac{2\zeta(k+a)(k+a)!}{(2\pi)^{k+a}(k+a)}<\frac{C_{3}(a)k!}{\pi^{k}}
\end{split}
\end{displaymath}
and
\begin{displaymath}
\begin{split}
\left|z^{a}U_{n+a}(z)\right|&\leq\frac{2\left|B_{n+a+1}\right|\sec^{n+a+2}\left(\frac{1}{2}\arg(z)\right)}{(n+a+1)|z|^{n+1}}\leq\frac{4\zeta(n+a+1)\sec^{n+a+2}\left(\frac{1}{2}\arg(z)\right)(n+a+1)!}{(2\pi)^{n+a+1}(n+a+1)|z|^{n+1}}\\
&<\frac{C_{4}(a)\sec^{n}\left(\frac{1}{2}\arg(z)\right)n!}{\pi^{n}|z|^{n+1}}
\end{split}
\end{displaymath}
for some positive constants $C_{3}(a)$, $C_{4}(a)$ depending on $a$ and independent of $n$.
\end{proof}

\newpage

\section{The generalized Faulhaber formulas}
\label{sec:The generalized Faulhaber formulas}

In this section we will prove our generalized versions of Faulhaber's formula, which all converge very rapidly. For their deductions, we will use the above Theorem \ref{Inverse factorial series expansions for zeta(s,z+1-y) and psi(z+1-y)}.

\begin{theorem}\label{extended generalized Faulhaber formulas}(extended generalized Faulhaber formulas)\\
For every complex number $m\in\mathbb{C}\setminus\{-1\}$ and every positive real number $x\in\mathbb{R}^{+}$, we have
\begin{equation}\label{Faulhaber's formula 1}
\begin{split}
\sum_{k=1}^{\lfloor x\rfloor}k^{m}=\frac{1}{m+1}x^{m+1}+\zeta\left(-m\right)+\frac{x^{m+1}}{m+1}\sum_{k=1}^{\infty}(-1)^{k}\frac{\sum_{l=1}^{k}{m+1\choose l}S^{(1)}_{k}(l)B_{l}(\{x\})}{(x+1)(x+2)\cdots(x+k)}.
\end{split}
\end{equation}
More generally, for every $x\in\mathbb{R}^{+}$ and every $a\in\mathbb{N}_{0}$, we have that
\begin{equation}\label{Faulhaber's formula 2}
\begin{split}
\sum_{k=1}^{\lfloor x\rfloor}k^{m}&=\frac{1}{m+1}x^{m+1}+\zeta\left(-m\right)+\frac{1}{m+1}\sum_{k=1}^{a}(-1)^{k}{m+1\choose k}B_{k}(\{x\})x^{m-k+1}\\
&\quad+\frac{x^{m-a+1}}{m+1}\sum_{k=1}^{\infty}(-1)^{k+a}\frac{\sum_{l=1}^{k}{m+1\choose l+a}S^{(1)}_{k}(l)B_{l+a}\left(\left\{x\right\}\right)}{(x+1)(x+2)\cdots(x+k)}
\end{split}
\end{equation}
and for $m=m_{1}+im_{2}\in\mathbb{C}\setminus\{-1\}$ with $m_{1}=\re(m)\geq-1$ the special case
\begin{equation}\label{Faulhaber's formula 3}
\begin{split}
\sum_{k=1}^{\lfloor x\rfloor}k^{m}&=\frac{1}{m+1}x^{m+1}+\zeta\left(-m\right)+\frac{1}{m+1}\sum_{k=1}^{\lfloor m_{1}+1\rfloor}(-1)^{k}{m+1\choose k}B_{k}(\{x\})x^{m-k+1}\\
&\quad+(-1)^{\lfloor m_{1}+1\rfloor}\frac{x^{m-\lfloor m_{1}+1\rfloor+1}}{m+1}\sum_{k=1}^{\infty}(-1)^{k}\frac{\sum_{l=1}^{k}{m+1\choose l+\lfloor m_{1}+1\rfloor}S^{(1)}_{k}(l)B_{l+\lfloor m_{1}+1\rfloor}(\{x\})}{(x+1)(x+2)\cdots(x+k)}.
\end{split}
\end{equation}
Moreover, if $m=-1$, we have for every positive real number $x\in\mathbb{R}^{+}$ and every $a\in\mathbb{N}_{0}$ that
\begin{equation}\label{Faulhaber's formula 4}
\begin{split}
\sum_{k=1}^{\lfloor x\rfloor}\frac{1}{k}&=\log(x)+\gamma-\sum_{k=1}^{a}\frac{B_{k}(\{x\})}{kx^{k}}+\frac{1}{x^{a}}\sum_{k=1}^{\infty}(-1)^{k+1}\frac{\sum_{l=1}^{k}\frac{(-1)^{l}}{l+a}S^{(1)}_{k}(l)B_{l+a}(\{x\})}{(x+1)(x+2)\cdots(x+k)}.
\end{split}
\end{equation}
In particular, we have for $x\in\mathbb{R}^{+}$ that
\begin{equation}\label{Faulhaber's formula 5}
\begin{split}
\sum_{k=1}^{\lfloor x\rfloor}\frac{1}{k}&=\log(x)+\gamma+\sum_{k=1}^{\infty}(-1)^{k+1}\frac{\sum_{l=1}^{k}\frac{(-1)^{l}}{l}S_{k}^{(1)}(l)B_{l}(\{x\})}{(x+1)(x+2)\cdots(x+k)}
\end{split}
\end{equation}
and that
\begin{equation}\label{Faulhaber's formula 6}
\begin{split}
\sum_{k=1}^{\lfloor x\rfloor}\frac{1}{k}&=\log(x)+\gamma-\frac{B_{1}(\{x\})}{x}+\sum_{k=1}^{\infty}(-1)^{k+1}\frac{\sum_{l=1}^{k}\frac{(-1)^{l}}{l+1}S_{k}^{(1)}(l)B_{l+1}(\{x\})}{x(x+1)(x+2)\cdots(x+k)}.
\end{split}
\end{equation}
\end{theorem}

\begin{proof}
From the formula \eqref{inverse factorial series expansion for the Hurwitz zeta function} with the parameters $s:=-m$, $z:=x$ and $y:=\{x\}$, we get
\begin{displaymath}
\begin{split}
\sum_{k=1}^{\left\lfloor x\right\rfloor}k^{m}-\zeta(-m)&=-\zeta(-m,x+1-\{x\})\\
&=\frac{1}{m+1}x^{m+1}+\frac{1}{m+1}\sum_{k=1}^{a}(-1)^{k}{m+1\choose k}B_{k}(\{x\})x^{m-k+1}\\
&\quad+\frac{x^{m-a+1}}{m+1}\sum_{k=1}^{\infty}(-1)^{k+a}\frac{\sum_{l=1}^{k}{m+1\choose l+a}S^{(1)}_{k}(l)B_{l+a}\left(\{x\}\right)}{(x+1)(x+2)\cdots(x+k)}
\end{split}
\end{displaymath}
by using the formula \eqref{Hurwitz summation formula} with $n:=\left\lfloor x\right\rfloor=x-\{x\}$ in the first step. This gives the above identity \eqref{Faulhaber's formula 2} with its special cases \eqref{Faulhaber's formula 1} and \eqref{Faulhaber's formula 3}.\\
Similarly, we now use the formula \eqref{inverse factorial series expansion for the digamma function} again with the variables $z:=x$ and $y:=\{x\}$, and then we get
\begin{displaymath}
\begin{split}
\sum_{k=1}^{\left\lfloor x\right\rfloor}\frac{1}{k}-\gamma&=\psi(x+1-\{x\})\\
&=\log(x)-\sum_{k=1}^{a}\frac{B_{k}(\{x\})}{kx^{k}}+\frac{1}{x^{a}}\sum_{k=1}^{\infty}(-1)^{k+1}\frac{\sum_{l=1}^{k}\frac{(-1)^{l}}{l+a}S^{(1)}_{k}(l)B_{l+a}(\{x\})}{(x+1)(x+2)\cdots(x+k)}
\end{split}
\end{displaymath}
by employing the formula \eqref{digamma summation formula} with $n:=\left\lfloor x\right\rfloor=x-\{x\}$ in the first line of the above calculation. This gives the above identity \eqref{Faulhaber's formula 4} with its special cases \eqref{Faulhaber's formula 5} and \eqref{Faulhaber's formula 6}.
\end{proof}

By setting $x:=n\in\mathbb{N}$ into Theorem \ref{extended generalized Faulhaber formulas}, we obtain the following

\begin{corollary}(generalized Faulhaber formulas)\\
For every complex number $m\in\mathbb{C}\setminus\{-1\}$ and every natural number $n\in\mathbb{N}$, we have
\begin{equation}
\begin{split}
\sum_{k=1}^{n}k^{m}=\frac{1}{m+1}n^{m+1}+\zeta\left(-m\right)+\frac{n^{m+1}}{m+1}\sum_{k=1}^{\infty}\frac{(-1)^{k}\sum_{l=1}^{k}{m+1\choose l}B_{l}S^{(1)}_{k}(l)}{(n+1)(n+2)\cdots(n+k)}
\end{split}
\end{equation}
and more generally when $a\in\mathbb{N}_{0}$ that
\begin{equation}
\begin{split}
\sum_{k=1}^{n}k^{m}&=\frac{1}{m+1}n^{m+1}+\zeta\left(-m\right)+\frac{1}{m+1}\sum_{k=1}^{a}(-1)^{k}{m+1\choose k}B_{k}n^{m-k+1}\\
&\quad+\frac{n^{m-a+1}}{m+1}\sum_{k=1}^{\infty}\frac{(-1)^{k+a}\sum_{l=1}^{k}{m+1\choose l+a}B_{l+a}S^{(1)}_{k}(l)}{(n+1)(n+2)\cdots(n+k)}.
\end{split}
\end{equation}
We have again when $m=m_{1}+im_{2}\in\mathbb{C}\setminus\{-1\}$ with $m_{1}=\re(m)\geq-1$ the special case
\begin{equation}
\begin{split}
\sum_{k=1}^{n}k^{m}&=\frac{1}{m+1}n^{m+1}+\zeta\left(-m\right)+\frac{1}{m+1}\sum_{k=1}^{\lfloor m_{1}+1\rfloor}(-1)^{k}{m+1\choose k}B_{k}n^{m-k+1}\\
&\quad+(-1)^{\lfloor m_{1}+1\rfloor}\frac{n^{m-\lfloor m_{1}+1\rfloor+1}}{m+1}\sum_{k=1}^{\infty}(-1)^{k}\frac{\sum_{l=1}^{k}{m+1\choose l+\lfloor m_{1}+1\rfloor}B_{l+\lfloor m_{1}+1\rfloor}S^{(1)}_{k}(l)}{(n+1)(n+2)\cdots(n+k)}.
\end{split}
\end{equation}
For $m=-1$, we have for every natural number $n\in\mathbb{N}$ and every $a\in\mathbb{N}_{0}$ that
\begin{equation}
\begin{split}
\sum_{k=1}^{n}\frac{1}{k}&=\log(n)+\gamma-\sum_{k=1}^{a}\frac{B_{k}}{kn^{k}}+\frac{1}{n^{a}}\sum_{k=1}^{\infty}\frac{(-1)^{k+1}\sum_{l=1}^{k}\frac{(-1)^{l}}{l+a}B_{l+a}S^{(1)}_{k}(l)}{(n+1)(n+2)\cdots(n+k)}.
\end{split}
\end{equation}
In particular, we have for every $n\in\mathbb{N}$ that
\begin{equation}
\begin{split}
\sum_{k=1}^{n}\frac{1}{k}&=\log(n)+\gamma+\sum_{k=1}^{\infty}\frac{(-1)^{k+1}\sum_{l=1}^{k}\frac{(-1)^{l}}{l}B_{l}S_{k}^{(1)}(l)}{(n+1)(n+2)\cdots(n+k)}\\
&=\log(n)+\gamma+\frac{1}{2(n+1)}+\frac{5}{12(n+1)(n+2)}+\frac{3}{4(n+1)(n+2)(n+3)}\\
&\quad+\frac{251}{120(n+1)(n+2)(n+3)(n+4)}+\ldots
\end{split}
\end{equation}
and that
\begin{equation}\label{GregorioFontanaExpansion}
\begin{split}
\sum_{k=1}^{n}\frac{1}{k}&=\log(n)+\gamma+\frac{1}{2n}+\sum_{k=1}^{\infty}\frac{(-1)^{k}\sum_{l=1}^{k}\frac{B_{l+1}}{l+1}S_{k}^{(1)}(l)}{n(n+1)(n+2)\cdots(n+k)}\\
&=\log(n)+\gamma+\frac{1}{2n}-\frac{1}{12n(n+1)}-\frac{1}{12n(n+1)(n+2)}-\frac{19}{120n(n+1)(n+2)(n+3)}\\
&\quad-\frac{9}{20n(n+1)(n+2)(n+3)(n+4)}-\ldots.
\end{split}
\end{equation}
\end{corollary}

\noindent For every positive real number $x\in\mathbb{R}^{+}$ and for every natural number $n\in\mathbb{N}$, we list the following $8$ most used generalized Faulhaber summation formulas:

\newpage

\begin{itemize}
\item[1.)]{{\bf Generalized Faulhaber summation formula for the partial sums of $\zeta(2)$:}\\
For every natural number $n\in\mathbb{N}$, we have that
\begin{equation}\label{StirlingExpansion}
\begin{split}
\sum_{k=1}^{n}\frac{1}{k^{2}}&=\zeta(2)-\frac{1}{n}+\sum_{k=1}^{\infty}\frac{(-1)^{k+1}\sum_{l=1}^{k}(-1)^{l}B_{l}S_{k}^{(1)}(l)}{n(n+1)(n+2)\cdots(n+k)}\\
&=\zeta(2)-\frac{1}{n}+\sum_{k=1}^{\infty}\frac{1}{k+1}\cdot\frac{(k-1)!}{n(n+1)(n+2)\cdots(n+k)}\\
&=\zeta(2)-\frac{1}{n}+\frac{1}{2n(n+1)}+\frac{1}{3n(n+1)(n+2)}+\frac{1}{2n(n+1)(n+2)(n+3)}\\
&\quad+\frac{6}{5n(n+1)(n+2)(n+3)(n+4)}+\ldots.
\end{split}
\end{equation}}

\item[2.)]{{\bf Extended generalized Faulhaber summation formula for the partial sums of $\zeta(3)$:}\\
For every real number $x\in\mathbb{R}^{+}$, we obtain
\begin{equation}
\begin{split}
\sum_{k=1}^{\lfloor x\rfloor}\frac{1}{k^{3}}&=\zeta(3)-\frac{1}{2x^{2}}+\frac{1}{2x}\sum_{k=1}^{\infty}(-1)^{k+1}\frac{\sum_{l=1}^{k}(-1)^{l}(l+1)S_{k}^{(1)}(l)B_{l}(\{x\})}{x(x+1)(x+2)\cdots(x+k)}.
\end{split}
\end{equation}}

\item[3.)]{{\bf Extended generalized Faulhaber summation formula for the sum of the square roots:}\\
For every real number $x\in\mathbb{R}^{+}$, we get
\begin{equation}
\begin{split}
\sum_{k=1}^{\lfloor x\rfloor}\sqrt{k}&=\frac{2}{3}x^{\frac{3}{2}}-\frac{1}{4\pi}\zeta\left(\frac{3}{2}\right)+x\sqrt{x}\sum_{k=1}^{\infty}(-1)^{k}\frac{\sum_{l=1}^{k}\frac{(-1)^{l}(2l-5)!!}{2^{l-1}l!}S_{k}^{(1)}(l)B_{l}(\{x\})}{(x+1)(x+2)\cdots(x+k)}.
\end{split}
\end{equation}}

\item[4.)]{{\bf Generalized Faulhaber summation formula for the partial sums of $\zeta(-3/2)$:}\\
For every natural number $n\in\mathbb{N}$, we have that
\begin{equation}
\begin{split}
\sum_{k=1}^{n}k\sqrt{k}&=\frac{2}{5}n^{\frac{5}{2}}+\frac{1}{2}n^{\frac{3}{2}}+\frac{1}{8}\sqrt{n}-\frac{3}{16\pi^{2}}\zeta\left(\frac{5}{2}\right)+3\sqrt{n}\sum_{k=1}^{\infty}(-1)^{k+1}\frac{\sum_{l=1}^{k}\frac{(2l-3)!!}{2^{l+1}(l+2)!}B_{l+2}S_{k}^{(1)}(l)}{(n+1)(n+2)\cdots(n+k)}\\
&=\frac{2}{5}n^{\frac{5}{2}}+\frac{1}{2}n^{\frac{3}{2}}+\frac{1}{8}\sqrt{n}-\frac{3}{16\pi^{2}}\zeta\left(\frac{5}{2}\right)+\frac{\sqrt{n}}{1920(n+1)(n+2)}+\frac{\sqrt{n}}{640(n+1)(n+2)(n+3)}\\
&\quad+\frac{611\sqrt{n}}{107520(n+1)(n+2)(n+3)(n+4)}+\frac{275\sqrt{n}}{10752(n+1)(n+2)(n+3)(n+4)(n+5)}\\
&\quad+\frac{159157\sqrt{n}}{1146880(n+1)(n+2)(n+3)(n+4)(n+5)(n+6)}+\ldots.
\end{split}
\end{equation}}

\item[5.)]{{\bf Generalized Faulhaber summation formula for the partial sums of $\zeta(-5/2)$:}\\
For every natural number $n\in\mathbb{N}$, we obtain that
\begin{equation}
\begin{split}
\sum_{k=1}^{n}k^{2}\sqrt{k}&=\frac{2}{7}n^{\frac{7}{2}}+\frac{1}{2}n^{\frac{5}{2}}+\frac{5}{24}n^{\frac{3}{2}}+\frac{15}{64\pi^{3}}\zeta\left(\frac{7}{2}\right)+15\sqrt{n}\sum_{k=1}^{\infty}(-1)^{k+1}\frac{\sum_{l=1}^{k}\frac{(2l-3)!!}{2^{l+2}(l+3)!}B_{l+3}S_{k}^{(1)}(l)}{(n+1)(n+2)\cdots(n+k)}\\
&=\frac{2}{7}n^{\frac{7}{2}}+\frac{1}{2}n^{\frac{5}{2}}+\frac{5}{24}n^{\frac{3}{2}}+\frac{15}{64\pi^{3}}\zeta\left(\frac{7}{2}\right)-\frac{\sqrt{n}}{384(n+1)}-\frac{\sqrt{n}}{384(n+1)(n+2)}\\
&\quad-\frac{37\sqrt{n}}{7168(n+1)(n+2)(n+3)}-\frac{55\sqrt{n}}{3584(n+1)(n+2)(n+3)(n+4)}\\
&\quad-\frac{1995\sqrt{n}}{32768(n+1)(n+2)(n+3)(n+4)(n+5)}-\ldots.
\end{split}
\end{equation}}

\item[6.)]{{\bf Generalized Faulhaber summation formula for the sum of the inverses of the square roots:}\\
For every natural number $n\in\mathbb{N}$, we get that
\begin{equation}
\begin{split}
\sum_{k=1}^{n}\frac{1}{\sqrt{k}}
&=2\sqrt{n}+\zeta\left(\frac{1}{2}\right)+\frac{1}{2\sqrt{n}}+\frac{1}{\sqrt{n}}\sum_{k=1}^{\infty}(-1)^{k}\frac{\sum_{l=1}^{k}\frac{(2l-1)!!}{2^{l}(l+1)!}B_{l+1}S_{k}^{(1)}(l)}{(n+1)(n+2)\cdots(n+k)}\\
&=2\sqrt{n}+\zeta\left(\frac{1}{2}\right)+\frac{1}{2\sqrt{n}}-\frac{1}{24\sqrt{n}(n+1)}-\frac{1}{24\sqrt{n}(n+1)(n+2)}\\
&\quad-\frac{31}{384\sqrt{n}(n+1)(n+2)(n+3)}-\frac{15}{64\sqrt{n}(n+1)(n+2)(n+3)(n+4)}-\ldots.
\end{split}
\end{equation}}

\item[7.)]{{\bf Extended generalized Faulhaber summation formula for the partial sums of $\zeta(3/2)$:}\\
For every real number $x\in\mathbb{R}^{+}$, we have
\begin{equation}
\begin{split}
\sum_{k=1}^{\lfloor x\rfloor}\frac{1}{k\sqrt{k}}
&=\zeta\left(\frac{3}{2}\right)-\frac{2}{\sqrt{x}}-\frac{B_{1}(\{x\})}{x\sqrt{x}}+\frac{2}{\sqrt{x}}\sum_{k=1}^{\infty}(-1)^{k+1}\frac{\sum_{l=1}^{k}\frac{(-1)^{l}(2l+1)!!}{2^{l+1}(l+1)!}S_{k}^{(1)}(l)B_{l+1}(\{x\})}{x(x+1)(x+2)\cdots(x+k)}.
\end{split}
\end{equation}}

\item[8.)]{{\bf Extended generalized Faulhaber summation formula for the partial sums of $\zeta(5/2)$:}\\
For every real number $x\in\mathbb{R}^{+}$, we obtain
\begin{equation}
\begin{split}
\sum_{k=1}^{\lfloor x\rfloor}\frac{1}{k^{2}\sqrt{k}}&=\zeta\left(\frac{5}{2}\right)-\frac{2}{3x^{\frac{3}{2}}}+\frac{4}{3\sqrt{x}}\sum_{k=1}^{\infty}(-1)^{k+1}\frac{\sum_{l=1}^{k}\frac{(-1)^{l}(2l+1)!!}{2^{l+1}l!}S_{k}^{(1)}(l)B_{l}(\{x\})}{x(x+1)(x+2)\cdots(x+k)}.
\end{split}
\end{equation}}
\end{itemize}

\noindent From Theorem \ref{Structure of inverse factorial series expansions} and the proof of Theorem \ref{extended generalized Faulhaber formulas}, it also follows
\begin{theorem}\label{other generalized Faulhaber formula versions}(other generalized Faulhaber formula versions)\\
For every $x\in\mathbb{R}^{+}$ and every $a\in\mathbb{N}_{0}$, we have that
\begin{equation}
\begin{split}
\sum_{k=1}^{\lfloor x\rfloor}k^{m}&=\frac{1}{m+1}x^{m+1}+\zeta\left(-m\right)+\frac{1}{m+1}\sum_{k=1}^{a}(-1)^{k}{m+1\choose k}B_{k}(\{x\})x^{m-k+1}\\
&\quad+\frac{x^{m+1}}{m+1}\sum_{k=1}^{\infty}(-1)^{k}\frac{\sum_{l=1}^{k}{m+1\choose l+a}S^{(1)}_{k}(l+a)B_{l+a}\left(\left\{x\right\}\right)}{(x+1)(x+2)\cdots(x+k)}
\end{split}
\end{equation}
and for $m=m_{1}+im_{2}\in\mathbb{C}\setminus\{-1\}$ with $m_{1}=\re(m)\geq-1$ the special case
\begin{equation}
\begin{split}
\sum_{k=1}^{\lfloor x\rfloor}k^{m}&=\frac{1}{m+1}x^{m+1}+\zeta\left(-m\right)+\frac{1}{m+1}\sum_{k=1}^{\lfloor m_{1}+1\rfloor}(-1)^{k}{m+1\choose k}B_{k}(\{x\})x^{m-k+1}\\
&\quad+\frac{x^{m+1}}{m+1}\sum_{k=1}^{\infty}(-1)^{k}\frac{\sum_{l=1}^{k}{m+1\choose l+\lfloor m_{1}+1\rfloor}S^{(1)}_{k}\left(l+\lfloor m_{1}+1\rfloor\right)B_{l+\lfloor m_{1}+1\rfloor}(\{x\})}{(x+1)(x+2)\cdots(x+k)}.
\end{split}
\end{equation}
If $m=-1$, we have for every positive real number $x\in\mathbb{R}^{+}$ and every $a\in\mathbb{N}_{0}$ that
\begin{equation}
\begin{split}
\sum_{k=1}^{\lfloor x\rfloor}\frac{1}{k}&=\log(x)+\gamma-\sum_{k=1}^{a}\frac{B_{k}(\{x\})}{kx^{k}}+\sum_{k=1}^{\infty}(-1)^{k+a+1}\frac{\sum_{l=1}^{k}\frac{(-1)^{l}}{l+a}S^{(1)}_{k}(l+a)B_{l+a}(\{x\})}{(x+1)(x+2)\cdots(x+k)}.
\end{split}
\end{equation}
In particular, we have for $x\in\mathbb{R}^{+}$ that
\begin{equation}
\begin{split}
\sum_{k=1}^{\lfloor x\rfloor}\frac{1}{k}&=\log(x)+\gamma-\frac{B_{1}(\{x\})}{x}+\sum_{k=1}^{\infty}(-1)^{k}\frac{\sum_{l=1}^{k}\frac{(-1)^{l}}{l+1}S_{k}^{(1)}(l+1)B_{l+1}(\{x\})}{(x+1)(x+2)\cdots(x+k)}\\
&=\log(x)+\gamma-\frac{\{x\}-\frac{1}{2}}{x}-\frac{\frac{1}{2}\{x\}^{2}-\frac{1}{2}\{x\}+\frac{1}{12}}{(x+1)(x+2)}-\frac{\frac{1}{3}\{x\}^{3}+\{x\}^{2}-\frac{4}{3}\{x\}+\frac{1}{4}}{(x+1)(x+2)(x+3)}\\
&\quad-\frac{\frac{1}{4}\{x\}^{4}+\frac{3}{2}\{x\}^{3}+\frac{11}{4}\{x\}^{2}-\frac{9}{2}\{x\}+\frac{109}{120}}{(x+1)(x+2)(x+3)(x+4)}-\ldots.
\end{split}
\end{equation}
\\
By setting $x:=n\in\mathbb{N}$, we obtain the following:\\
\\
For every $n\in\mathbb{N}$ and every $a\in\mathbb{N}_{0}$, we have that
\begin{equation}
\begin{split}
\sum_{k=1}^{n}k^{m}&=\frac{1}{m+1}n^{m+1}+\zeta\left(-m\right)+\frac{1}{m+1}\sum_{k=1}^{a}(-1)^{k}{m+1\choose k}B_{k}n^{m-k+1}\\
&\quad+\frac{n^{m+1}}{m+1}\sum_{k=1}^{\infty}(-1)^{k}\frac{\sum_{l=1}^{k}{m+1\choose l+a}B_{l+a}S^{(1)}_{k}(l+a)}{(n+1)(n+2)\cdots(n+k)}
\end{split}
\end{equation}
and for $m=m_{1}+im_{2}\in\mathbb{C}\setminus\{-1\}$ with $m_{1}=\re(m)\geq-1$ the special case
\begin{equation}
\begin{split}
\sum_{k=1}^{n}k^{m}&=\frac{1}{m+1}n^{m+1}+\zeta\left(-m\right)+\frac{1}{m+1}\sum_{k=1}^{\lfloor m_{1}+1\rfloor}(-1)^{k}{m+1\choose k}B_{k}n^{m-k+1}\\
&\quad+\frac{n^{m+1}}{m+1}\sum_{k=1}^{\infty}(-1)^{k}\frac{\sum_{l=1}^{k}{m+1\choose l+\lfloor m_{1}+1\rfloor}B_{l+\lfloor m_{1}+1\rfloor}S^{(1)}_{k}\left(l+\lfloor m_{1}+1\rfloor\right)}{(n+1)(n+2)\cdots(n+k)}.
\end{split}
\end{equation}
If $m=-1$, we have for every positive real number $n\in\mathbb{N}$ and every $a\in\mathbb{N}_{0}$ that
\begin{equation}
\begin{split}
\sum_{k=1}^{n}\frac{1}{k}&=\log(n)+\gamma-\sum_{k=1}^{a}\frac{B_{k}}{kn^{k}}+\sum_{k=1}^{\infty}(-1)^{k+a+1}\frac{\sum_{l=1}^{k}\frac{(-1)^{l}}{l+a}B_{l+a}S^{(1)}_{k}(l+a)}{(n+1)(n+2)\cdots(n+k)}.
\end{split}
\end{equation}
In particular, we have for $n\in\mathbb{N}$ that
\begin{equation}
\begin{split}
\sum_{k=1}^{n}\frac{1}{k}&=\log(n)+\gamma+\frac{1}{2n}+\sum_{k=1}^{\infty}(-1)^{k}\frac{\sum_{l=1}^{k}\frac{(-1)^{l}}{l+1}B_{l+1}S_{k}^{(1)}(l+1)}{(n+1)(n+2)\cdots(n+k)}\\
&=\log(n)+\gamma+\frac{1}{2n}-\frac{1}{12(n+1)(n+2)}-\frac{1}{4(n+1)(n+2)(n+3)}\\
&\quad-\frac{109}{120(n+1)(n+2)(n+3)(n+4)}-\ldots.
\end{split}
\end{equation}
\end{theorem}

\section{Conclusion}
\label{sec:Conclusion}

We have proved a rapidly convergent generalization of Faulhaber's formula to sums of arbitrary complex powers $m\in\mathbb{C}$. In our eyes, these formulas are useful because of their rapid convergence. We believe that they will also have applications in physics \cite{18} such as the extended version of Faulhaber's formula \cite{19,20}. With the universal technique, explained in this paper, one can obtain other summation formulas of this type \cite{21,22}, as for example with Theorem \ref{Structure of inverse factorial series expansions} and equation \eqref{LogGammaExpansion} we obtain:\\
\\
\noindent{\bf Generalized convergent Stirling summation formulas for the sums $\sum_{k=1}^{\left\lfloor x\right\rfloor}\ln(k)$ and $\sum_{k=1}^{n}\ln(k)$:}\\
For every real number $x\in\mathbb{R}^{+}$ and every natural number $a\in\mathbb{N}_{0}$, we have that
\begin{equation}\label{StirlingExpansion1}
\begin{split}
\sum_{k=1}^{\left\lfloor x\right\rfloor}\ln(k)&=x\ln(x)-x+\frac{1}{2}\ln(2\pi)-\ln(x)B_{1}(\{x\})+\sum_{k=1}^{a}\frac{B_{k+1}(\{x\})}{k(k+1)x^{k}}\\
&\quad+\frac{1}{x^{a}}\sum_{k=1}^{\infty}(-1)^{k}\frac{\sum_{l=1}^{k}\frac{(-1)^{l}S_{k}^{(1)}(l)}{(l+a)(l+a+1)}B_{l+a+1}(\{x\})}{(x+1)(x+2)\cdots(x+k)},
\end{split}
\end{equation}
or
\begin{equation}
\begin{split}
\sum_{k=1}^{\left\lfloor x\right\rfloor}\ln(k)&=x\ln(x)-x+\frac{1}{2}\ln(2\pi)-\ln(x)B_{1}(\{x\})+\sum_{k=1}^{a}\frac{B_{k+1}(\{x\})}{k(k+1)x^{k}}\\
&\quad+\sum_{k=1}^{\infty}(-1)^{k+a}\frac{\sum_{l=1}^{k}\frac{(-1)^{l}S_{k}^{(1)}(l+a)}{(l+a)(l+a+1)}B_{l+a+1}(\{x\})}{(x+1)(x+2)\cdots(x+k)},
\end{split}
\end{equation}
or
\begin{equation}
\begin{split}
\sum_{k=1}^{\left\lfloor x\right\rfloor}\ln(k)&=x\ln(x)-x+\frac{1}{2}\ln(2\pi)-\ln(x)B_{1}(\{x\})+\sum_{k=1}^{a}\frac{B_{k+1}(\{x\})}{k(k+1)x^{k}}\\
&\quad+x\sum_{k=1}^{\infty}(-1)^{k+a+1}\frac{\sum_{l=1}^{k}\frac{(-1)^{l}S_{k}^{(1)}(l+a+1)}{(l+a)(l+a+1)}B_{l+a+1}(\{x\})}{(x+1)(x+2)\cdots(x+k)}
\end{split}
\end{equation}
and for every natural number $n\in\mathbb{N}$ and every natural number $a\in\mathbb{N}_{0}$ that
\begin{equation}\label{StirlingExpansion2}
\begin{split}
\sum_{k=1}^{n}\ln(k)&=n\ln(n)-n+\frac{1}{2}\ln(2\pi)+\frac{1}{2}\ln(n)+\sum_{k=1}^{a}\frac{B_{k+1}}{k(k+1)n^{k}}\\
&\quad+\frac{1}{n^{a}}\sum_{k=1}^{\infty}(-1)^{k}\frac{\sum_{l=1}^{k}(-1)^{l}\frac{B_{l+a+1}S_{k}^{(1)}(l)}{(l+a)(l+a+1)}}{(n+1)(n+2)\cdots(n+k)},
\end{split}
\end{equation}
or
\begin{equation}
\begin{split}
\sum_{k=1}^{n}\ln(k)&=n\ln(n)-n+\frac{1}{2}\ln(2\pi)+\frac{1}{2}\ln(n)+\sum_{k=1}^{a}\frac{B_{k+1}}{k(k+1)n^{k}}\\
&\quad+\sum_{k=1}^{\infty}(-1)^{k+a}\frac{\sum_{l=1}^{k}\frac{(-1)^{l}S_{k}^{(1)}(l+a)}{(l+a)(l+a+1)}B_{l+a+1}}{(n+1)(n+2)\cdots(n+k)},
\end{split}
\end{equation}
or
\begin{equation}
\begin{split}
\sum_{k=1}^{n}\ln(k)&=n\ln(n)-n+\frac{1}{2}\ln(2\pi)+\frac{1}{2}\ln(n)+\sum_{k=1}^{a}\frac{B_{k+1}}{k(k+1)n^{k}}\\
&\quad+n\sum_{k=1}^{\infty}(-1)^{k+a+1}\frac{\sum_{l=1}^{k}\frac{(-1)^{l}S_{k}^{(1)}(l+a+1)}{(l+a)(l+a+1)}B_{l+a+1}}{(n+1)(n+2)\cdots(n+k)}.
\end{split}
\end{equation}
For many other functions $f(t)$, we can prove with Theorem \ref{Structure of inverse factorial series expansions} the following summation formulas:\\
\\
{\bf Convergent version of the Euler-Maclaurin summation formula:}\\
For many functions $f(t)$ it holds the following: For every real number $x\in\mathbb{R}^{+}$ and every $a\in\mathbb{N}_{0}$, we have
\begin{equation}
\begin{split}
\sum_{k=1}^{\lfloor x\rfloor}f(k)&=\int_{1}^{x}f(t)dt+C_{f}+\sum_{k=1}^{a}(-1)^{k}\frac{B_{k}(\{x\})}{k!}f^{(k-1)}(x)\\
&\quad+\sum_{k=1}^{\infty}(-1)^{k+a}\frac{\sum_{l=1}^{k}\frac{S^{(1)}_{k}(l)}{(l+a)!}f^{(l+a-1)}(x)B_{l+a}(\{x\})x^{l}}{(x+1)(x+2)\cdots(x+k)},
\end{split}
\end{equation}
or
\begin{equation}
\begin{split}
\sum_{k=1}^{\lfloor x\rfloor}f(k)&=\int_{1}^{x}f(t)dt+C_{f}+\sum_{k=1}^{a}(-1)^{k}\frac{B_{k}(\{x\})}{k!}f^{(k-1)}(x)\\
&\quad+x^{a}\sum_{k=1}^{\infty}(-1)^{k}\frac{\sum_{l=1}^{k}\frac{S^{(1)}_{k}(l+a)}{(l+a)!}f^{(l+a-1)}(x)B_{l+a}(\{x\})x^{l}}{(x+1)(x+2)\cdots(x+k)}
\end{split}
\end{equation}
and for $n\in\mathbb{N}$ and $a\in\mathbb{N}_{0}$ we have
\begin{equation}
\begin{split}
\sum_{k=1}^{n}f(k)&=\int_{1}^{n}f(t)dt+C_{f}+\sum_{k=1}^{a}(-1)^{k}\frac{B_{k}}{k!}f^{(k-1)}(n)\\
&\quad+\sum_{k=1}^{\infty}(-1)^{k+a}\frac{\sum_{l=1}^{k}\frac{S^{(1)}_{k}(l)}{(l+a)!}f^{(l+a-1)}(n)B_{l+a}n^{l}}{(n+1)(n+2)\cdots(n+k)},
\end{split}
\end{equation}
or
\begin{equation}
\begin{split}
\sum_{k=1}^{n}f(k)&=\int_{1}^{n}f(t)dt+C_{f}+\sum_{k=1}^{a}(-1)^{k}\frac{B_{k}}{k!}f^{(k-1)}(n)\\
&\quad+n^{a}\sum_{k=1}^{\infty}(-1)^{k}\frac{\sum_{l=1}^{k}\frac{S^{(1)}_{k}(l+a)}{(l+a)!}f^{(l+a-1)}(n)B_{l+a}n^{l}}{(n+1)(n+2)\cdots(n+k)},
\end{split}
\end{equation}
where the constant $C_{f}$ is given by
\begin{equation}
\begin{split}
C_{f}&=f(1)-\sum_{k=1}^{a}(-1)^{k}\frac{B_{k}}{k!}f^{(k-1)}(1)-\sum_{k=1}^{\infty}(-1)^{k+a}\frac{\sum_{l=1}^{k}\frac{S^{(1)}_{k}(l)}{(l+a)!}f^{(l+a-1)}(1)B_{l+a}}{(k+1)!}\;\;\forall a\in\mathbb{N}_{0}.
\end{split}
\end{equation}

\newpage

\noindent{\bf Convergent version of the Boole summation formula:}\\
For many functions $f(t)$  it holds the following: For every real number $x\in\mathbb{R}^{+}$ and every $a\in\mathbb{N}_{0}$, we have
\begin{equation}
\begin{split}
\sum_{k=1}^{\lfloor x\rfloor}(-1)^{k+1}f(k)&=C_{f}+\frac{(-1)^{x-\{x\}}}{2}\sum_{k=1}^{a}(-1)^{k}\frac{E_{k-1}(\{x\})}{(k-1)!}f^{(k-1)}(x)\\
&\quad+\frac{(-1)^{x-\{x\}}}{2}\sum_{k=1}^{\infty}(-1)^{k+a}\frac{\sum_{l=1}^{k}\frac{S^{(1)}_{k}(l)}{(l+a-1)!}f^{(l+a-1)}(x)E_{l+a-1}(\{x\})x^{l}}{(x+1)(x+2)\cdots(x+k)},
\end{split}
\end{equation}
or
\begin{equation}
\begin{split}
\sum_{k=1}^{\lfloor x\rfloor}(-1)^{k+1}f(k)&=C_{f}+\frac{(-1)^{x-\{x\}}}{2}\sum_{k=1}^{a}(-1)^{k}\frac{E_{k-1}(\{x\})}{(k-1)!}f^{(k-1)}(x)\\
&\quad+\frac{(-1)^{x-\{x\}}}{2}x^{a}\sum_{k=1}^{\infty}(-1)^{k}\frac{\sum_{l=1}^{k}\frac{S^{(1)}_{k}(l+a)}{(l+a-1)!}f^{(l+a-1)}(x)E_{l+a-1}(\{x\})x^{l}}{(x+1)(x+2)\cdots(x+k)},
\end{split}
\end{equation}
where $E_{n}(\{x\})$ denotes the fractional Euler polynomials \cite{22} and for $n\in\mathbb{N}$ and $a\in\mathbb{N}_{0}$ we have
\begin{equation}
\begin{split}
\sum_{k=1}^{n}(-1)^{k+1}f(k)&=C_{f}+(-1)^{n+1}\sum_{k=1}^{a}(-1)^{k}\frac{(2^{k}-1)B_{k}}{k!}f^{(k-1)}(n)\\
&\quad+(-1)^{n+1}\sum_{k=1}^{\infty}(-1)^{k+a}\frac{\sum_{l=1}^{k}\frac{S^{(1)}_{k}(l)}{(l+a)!}(2^{l+a}-1)f^{(l+a-1)}(n)B_{l+a}n^{l}}{(n+1)(n+2)\cdots(n+k)},
\end{split}
\end{equation}
or
\begin{equation}
\begin{split}
\sum_{k=1}^{n}(-1)^{k+1}f(k)&=C_{f}+(-1)^{n+1}\sum_{k=1}^{a}(-1)^{k}\frac{(2^{k}-1)B_{k}}{k!}f^{(k-1)}(n)\\
&\quad+(-1)^{n+1}n^{a}\sum_{k=1}^{\infty}(-1)^{k}\frac{\sum_{l=1}^{k}\frac{S^{(1)}_{k}(l+a)}{(l+a)!}(2^{l+a}-1)f^{(l+a-1)}(n)B_{l+a}n^{l}}{(n+1)(n+2)\cdots(n+k)},
\end{split}
\end{equation}
where the constant $C_{f}$ is given by
\begin{equation}
\begin{split}
C_{f}&=f(1)-\sum_{k=1}^{a}(-1)^{k}\frac{(2^{k}-1)B_{k}}{k!}f^{(k-1)}(1)-\sum_{k=1}^{\infty}(-1)^{k+a}\frac{\sum_{l=1}^{k}\frac{(2^{l+a}-1)S^{(1)}_{k}(l)}{(l+a)!}f^{(l+a-1)}(1)B_{l+a}}{(k+1)!}\;\;\forall a\in\mathbb{N}_{0}.
\end{split}
\end{equation}
A generalization of Faulhaber's formula for alternating sums can be found in \cite{22} and an extended form of it is given for $x\in\mathbb{R}^{+}$ by:\\
{\bf Alternating versions of Faulhaber's formula:}\\
For every $x\in\mathbb{R}^{+}$, it is given by
\begin{equation}
\begin{split}
\sum_{k=1}^{\lfloor x\rfloor}(-1)^{k+1}k^{m}&=\eta(-m)+\frac{(-1)^{x-\{x\}}}{2}\sum_{k=1}^{m+1}(-1)^{k}{m\choose k-1}E_{k-1}(\{x\})x^{m-k-1}\;\;\forall m\in\mathbb{N}_{0}
\end{split}
\end{equation}
and for $n\in\mathbb{N}$ by
\begin{equation}
\begin{split}
\sum_{k=1}^{n}(-1)^{k+1}k^{m}&=\eta(-m)+(-1)^{n+1}\sum_{k=1}^{m+1}(-1)^{k}\frac{2^{k}-1}{k}{m\choose k-1}B_{k}n^{m-k-1}\;\;\forall m\in\mathbb{N}_{0},
\end{split}
\end{equation}
as well as for $x\in\mathbb{R}^{+}$ and $a\in\mathbb{N}_{0}$ by
\begin{equation}
\begin{split}
\sum_{k=1}^{\lfloor x\rfloor}(-1)^{k+1}k^{m}&=\eta(-m)+\frac{(-1)^{x-\{x\}}}{2}\sum_{k=1}^{a}(-1)^{k}{m\choose k-1}E_{k-1}(\{x\})x^{m-k-1}\\
&\quad+\frac{(-1)^{x-\{x\}}x^{m-a-1}}{2}\sum_{k=1}^{\infty}(-1)^{k+a}\frac{\sum_{l=1}^{k}{m\choose l+a-1}S^{(1)}_{k}(l)E_{l+a-1}(\{x\})}{(x+1)(x+2)\cdots(x+k)}\;\;\forall m\in\mathbb{C},
\end{split}
\end{equation}
or
\begin{equation}
\begin{split}
\sum_{k=1}^{\lfloor x\rfloor}(-1)^{k+1}k^{m}&=\eta(-m)+\frac{(-1)^{x-\{x\}}}{2}\sum_{k=1}^{a}(-1)^{k}{m\choose k-1}E_{k-1}(\{x\})x^{m-k-1}\\
&\quad+\frac{(-1)^{x-\{x\}}x^{m-1}}{2}\sum_{k=1}^{\infty}(-1)^{k}\frac{\sum_{l=1}^{k}{m\choose l+a-1}S^{(1)}_{k}(l+a)E_{l+a-1}(\{x\})}{(x+1)(x+2)\cdots(x+k)}\;\;\forall m\in\mathbb{C}
\end{split}
\end{equation}
and for $n\in\mathbb{N}$ and $a\in\mathbb{N}_{0}$ by
\begin{equation}
\begin{split}
\sum_{k=1}^{n}(-1)^{k+1}k^{m}&=\eta(-m)+(-1)^{n+1}\sum_{k=1}^{a}(-1)^{k}\frac{2^{k}-1}{k}{m\choose k-1}B_{k}n^{m-k-1}\\
&\quad+(-1)^{n+1}n^{m-a-1}\sum_{k=1}^{\infty}(-1)^{k+a}\frac{\sum_{l=1}^{k}\frac{2^{l+a}-1}{l+a}{m\choose l+a-1}B_{l+a}S^{(1)}_{k}(l)}{(n+1)(n+2)\cdots(n+k)}\;\;\forall m\in\mathbb{C},
\end{split}
\end{equation}
or
\begin{equation}
\begin{split}
\sum_{k=1}^{n}(-1)^{k+1}k^{m}&=\eta(-m)+(-1)^{n+1}\sum_{k=1}^{a}(-1)^{k}\frac{2^{k}-1}{k}{m\choose k-1}B_{k}n^{m-k-1}\\
&\quad+(-1)^{n+1}n^{m-1}\sum_{k=1}^{\infty}(-1)^{k}\frac{\sum_{l=1}^{k}\frac{2^{l+a}-1}{l+a}{m\choose l+a-1}B_{l+a}S^{(1)}_{k}(l+a)}{(n+1)(n+2)\cdots(n+k)}\;\;\forall m\in\mathbb{C}.
\end{split}
\end{equation}
In the above six equations $\eta(s):=\sum_{k=1}^{\infty}\frac{(-1)^{k+1}}{k^{s}}$ denotes the Dirichlet eta function.

\section{Acknowledgment}
\label{sec:Acknowledgement}

This work was supported by SNSF
(Swiss National Science Foundation) under grant 169247.

\bigskip
\hrule
\bigskip

\noindent 2010 {\it Mathematics Subject Classification}: Primary 65B15; Secondary 11B68.

\noindent\emph{Keywords: }generalization of Faulhaber's formula, extended Faulhaber formula,
finite Weniger transformation, Stirling number of the first kind, Bernoulli polynomial, Bernoulli number, generalized convergent Stirling summation formula, alternating Faulhaber formula.


\begin{thebibliography}{9}

\bibitem{1}
Kevin J. McGown and Harold R. Parks, The generalization of Faulhaber's formula to sums of non-integral powers, \emph{J. Math. Anal. Appl.} {\bf330} (2007), 571--575.

\bibitem{2}
Ernst Joachim Weniger, Summation of divergent power series by means of factorial series, \emph{Appl. Num. Math.} {\bf60} (2010), 1429--1441.

\bibitem{3}
Oskar Schl\"omilch, Ueber Facult\"atenreihen, \emph{Zeitschrift f\"ur Mathematik und Physik} {\bf 4} (1859), 390--415.

\bibitem{4}
Oskar Schl\"omilch, Relationen zwischen den Facult\"atencoefficienten, \emph{Arch. für Math. und Phys.} {\bf 9} (1847), 333--335.

\bibitem{5}
Oskar Schl\"omilch, Über die independente Bestimmung der Co\"efficienten unendlicher Reihen und der Facult\"atencoefficienten insbesondere, \emph{Arch. für Math. und Phys.} {\bf 18} (1852), 306--327.

\bibitem{6}
I. Tweedle, \emph{James Stirling's Methodus Differentialis: An Annotated Translation of Stirling's Text}, Springer--Verlag, London, 2003.

\bibitem{7}
A. C. Aitken, A note on inverse central factorial series, \emph{Proc. Edinburgh Math. Soc. (2)}, {\bf Vol. 7, No. 03} (1946).

\bibitem{8}
Donatella Merlini, Renzo Sprugnoli, and M. Cecilia Verri, The Cauchy numbers, \emph{Discrete Math. (Science Direct)} {\bf 306} (2006), 1906--1920.

\bibitem{9}
Frank W. J. Olver, Daniel W. Lozier, Ronald F. Boisvert, and Charles W. Clark, \emph{NIST Handbook of Mathematical Functions}, Cambridge University Press, 2010.

\bibitem{10}
N. J. A. Sloane, \emph{Online Encyclopedia of Integer Sequences}. Published electronically at \url{https://oeis.org}, 2021.

\bibitem{11}
Yudell L. Luke, \emph{
The special functions and their approximations}, Vol. I, 
Academic Press, New York-London, 1969.

\bibitem{12}
Victor Kac and Pokman Cheung, \emph{Quantum Calculus}, Springer, 2002.

\bibitem{13}
Hugh L. Montgomery and Robert C. Vaughan, \emph{Multiplicative number theory. I. Classical theory}, Cambridge University Press, 2007.

\bibitem{14}
Frank W. J. Olver, \emph{Asymptotics and special functions}, 
Reprint of the 1974 original, Wellesley, MA: A. K. Peters Ltd., 1997.

\bibitem{15}
G. Nemes, 
Error bounds for the asymptotic expansion of the Hurwitz zeta function, \emph{
Proc. A.} {\bf473} (2017), 1--16.

\bibitem{16}
G. N. Watson, The transformation of an asymptotic series into a convergent series of inverse factorials, \emph{Rend. Circ. Mat. Palermo} {\bf 34} (1912), 41--88.

\bibitem{17}
Niels\;Nielsen, \emph{Handbuch der Theorie der Gammafunktion}, Teubner, Leipzig and Berlin, 1906.

\bibitem{18}
Hanno Sahlmann and Thomas Zilker, Quantum surface holonomies for loop quantum gravity and their application to black hole horizons, \emph{Phys. Rev. D} {\bf 102, no. 2} (2020), 026009, 33 pp.

\bibitem{19}
Ant\^onio F. Neto, A note on a theorem of Schumacher, \emph{J. Integer Seq.} {\bf 19} (2016), \href{https://cs.uwaterloo.ca/journals/JIS/VOL19/Neto2/neto21.pdf}{Article 16.8.5}.

\bibitem{20}
Raphael Schumacher, An extended version of Faulhaber's formula, \emph{J. Integer Seq.} {\bf 19} (2016), \href{https://cs.uwaterloo.ca/journals/JIS/VOL19/Schumacher/schu3.pdf}{Article 16.4.2}.

\bibitem{21}
Raphael Schumacher, Rapidly convergent summation formulas involving Stirling series, arXiv:1602.00336v1 [math.NT], (2016), \url{https://arxiv.org/abs/1602.00336}.

\bibitem{22}
Raphael Schumacher, Extension of summation formulas involving Stirling series, arXiv:1605.09204v1 [math.NT], (2016), \url{https://arxiv.org/abs/1605.09204}.

\end{thebibliography}
\end{document}